\documentclass[9pt]{amsart}

\usepackage{amssymb}
\usepackage{array}
\usepackage[all,cmtip]{xy}

\title{Tannakization of quasi-categories and monadic descent}
\author{Romie Banerjee}
\date{\today}
\email{romie@iiserb.ac.in}

\begin{document}

\begin{abstract}
Given a symmetric monoidal stable $\infty$-category $\mathcal{C}$ and a left adjoint symmetric monoidal fiber functor to $\operatorname{Mod}_A^{\otimes}$ for some $\mathbb{E}_{\infty}$-ring $A$, one can construct a derived group scheme $G$ of monoidal automorphisms of this functor. The left adjoint fiber functor also induces a monad on $\mathcal{C}$. Under some finiteness hypothesis on the fiber functor, we show there is a comparison functor from the category of representations of $G$ to the descent category of the induced monad on $\mathcal{C}$. 
\end{abstract}

\maketitle
\setcounter{tocdepth}{1}
\tableofcontents

\theoremstyle{definition}
\newtheorem{definition}{Definition}[section]
\newtheorem{example}{Example}[section]
\newtheorem{quest}{Question}[section]

\theoremstyle{plain}
\newtheorem{prop}{Proposition}[section]
\newtheorem{lemma}{Lemma}[section]

\newtheorem{theorem}{Theorem}[section]

\newtheorem{Theorem}{Theorem}
\renewcommand*{\theTheorem}{\Alph{Theorem}}

\newtheorem{cor}{Corollary}[section]

\theoremstyle{remark}
\newtheorem{remark}{Remark}[section]
\newtheorem{nota}{Notation}[section]

\section{Introduction}

\subsection{Tannakization and motivic Galois group}
Given a symmetric monoidal stable $\infty$-category $\mathcal{C}^{\otimes}$, an $\mathbb{E}_{\infty}$-ring $A$ and a fiber functor (symmetric monoidal, left exact) $\omega:\mathcal{C}^{\otimes} \to \operatorname{Mod}_A^{\otimes}$, one can associate a derived group scheme $G$ over $A$, called the {\em tannakization} of the data $(\mathcal{C}^{\otimes}, A, \omega)$. The derived group scheme $G$ agrees with the group scheme $\operatorname{Aut}(\omega)$ of (higher) monoidal automorphisms of the fiber functor $\omega$. (see \cite{Iwanari})

The fiber functor lifts to a symmetric monoidal map into the $\infty$-category of representations of $G$ making the following diagram commute in the $\infty$-category of symmetric monoidal $\infty$-categories.
$$\xymatrix{
&\operatorname{Rep}_A(G)^{\otimes} \ar[d]^{\operatorname{forget}} \\
\mathcal{C}^{\otimes} \ar[ur] \ar[r]_{\omega} & \operatorname{Mod}_A^{\otimes}\\
}$$ 
Here $\operatorname{Rep}_A(G)$ is the $\infty$-category of $A$-modules with a $G$-action. 

It may happen that the tannakization is representable by a derived Hopf algebra $\mathcal{H}$ over $A$ and the lift identifies $\mathcal{C}^{\otimes}$ with the category of representations of $G$. In this case the affine group scheme $\mathsf{Spec}\,\mathcal{H}$ may be called the {\em motivic Galois group} of the category $\mathcal{C}^{\otimes}$ at the basepoint $\omega$. 

\subsubsection{Example} As a prototypical example one can consider a fiber functor which is a symmetric monoidal left adjoint map $\omega:B\operatorname{-mod}^{\otimes} \to A\operatorname{-mod}^{\otimes}$ for $\mathbb{E}_{\infty}$-rings $A$ and $B$. Then $\omega = f^*$ for some $\mathbb{E}_{\infty}$-algebra map $f:B\to A$. The tannakization is the derived affine group scheme $G = \mathsf{Spec}(A\otimes_BA)$ over $\mathsf{Spec}A$. The group structure arises from the Bar construction associated with the map $\mathsf{Spec}B \to \mathsf{Spec}A$.

The category of $G$-representations is equivalent to the category of descent data for the map $f:B\to A$, and $G$ is the motivic Galois group of $\operatorname{Mod}_B^{\otimes}$ if and only if $f$ is of effective descent for modules.

\subsection{The descent category}
Let $T$ be a monad on an $\infty$-category $\mathcal{C}$. The descent category of $T$ in $\mathcal{C}$ a category of comodules in the $\infty$-category of $T$-modules in $\mathcal{C}$. Informally, the objects of $\operatorname{Desc}_{\mathcal{C}}(T)$ are $T$-modules + ``descent data". There is a canonical map $Q_T:\mathcal{C} \to \operatorname{Desc}_{\mathcal{C}}(T)$. The monad $T$ is said to be of {\em effective descent in $\mathcal{C}$} when $Q_T$ is an equivalence. The descent category is said to be {\em Tannakian} if it is equivalent to a category of comodules over a coalgebra in a monoidal $\infty$-category. (see \cite{M},\cite{Hess})

\subsubsection{Example} Let $\phi:B\to A$ be a map of $\mathbb{E}_{\infty}$-rings. The induced adjunction $-\otimes_AB: \operatorname{Mod}_A \rightleftarrows \operatorname{Mod}_B: \phi_*$ defines a monad $T = \phi_*(-\otimes_AB)$ on $\operatorname{Mod}_A$ and a comonad $K=\phi_*(-)\otimes_AB$ on $\operatorname{Mod}_B$. There is a coalgebra object $A\otimes_BA$ over $A$ whose underlying $A$-module is $A\otimes_BA$ and coalgebra structure is given by the cosimplicial $\mathbb{E}_{\infty}$-ring over $A$ coming from the cobar construction associated with $\phi:A\to B$. The descent category for the monad $T$ is equivalent to the category of comodules over this coalgebra.

\subsection{Main results} 
The aim of this paper is to relate these two categories in the context of a fiber functor adjunction. Let $\mathcal{C}^{\otimes}$ be a symmetric monoidal stable $\infty$-category and $\omega:\mathcal{C}^{\otimes} \rightleftarrows \operatorname{Mod}_A^{\otimes}:\omega'$ be an adjunction induced by a fiber functor $\omega$. If we assume $\omega$ to preserve limits then $\omega \simeq \operatorname{Hom}_{\mathcal{C}}(X_{\omega},-)$ where $X_{\omega}$ is a compact object in $\operatorname{Mod}_A(\mathcal{C})$. In fact there is an equivalence between the $\infty$-category of limit preserving  $\operatorname{Mod}_A^{\otimes}$ valued fiber functors on $\mathcal{C}^{\otimes}$ and the $\infty$-category of compact $A$-modules in $\mathcal{C}$. This is used to prove that the tannakization $\operatorname{Aut}_{\omega}^{\otimes}$ in this case is represented by a derived {\em affine} group scheme. Also in this case there is a comparison functor $$\operatorname{Rep}_A(\operatorname{Aut}_{\omega}^{\otimes}) \to \operatorname{Desc}_T(\mathcal{C}).$$

There is a notion of a tensor product of presentable $\infty$-categories, generalizing Deligne's tensor product of abelian categories, that makes the $\infty$-category of presentable $\infty$-categories into a symmetric monoidal $\infty$-category. A symmetric monoidal presentable $\infty$-category is a commutative monoid objects with respect to this tensor product. The symmetric monoidal fiber functor $\omega:\mathcal{C}^{\otimes} \to \operatorname{Mod}^{\otimes}_A$ exhibits $\operatorname{Mod}_A$ as a commutative algebra over $\mathcal{C}^{\otimes}$ in the $\infty$-category of presentable $\infty$-categories. We consider $\varprojlim \mathcal{C}^{\bullet}$, where $\mathcal{C}^{\bullet}$ is the cobar construction associated with the map $\omega$. This acts as a bridge to go between $\operatorname{Rep}_A(G)$ and $\operatorname{Desc}_T(\mathcal{C})$. There are maps

$$\xymatrix{
&\operatorname{Rep}_A(G) \\
\varprojlim \mathcal{C}^{\bullet} \ar[ur] \ar[dr] \\
&\operatorname{Desc}_T(\mathcal{C})\\
}$$

\begin{theorem}\label{theorem}
Let $\mathcal{C}^{\otimes}$ be a symmetric monoidal stable $\infty$-category, $A$ an $\mathbb{E}_{\infty}$-ring spectrum and $\omega\in \operatorname{Fun}^{\otimes,L}(\mathcal{C}^{\otimes},\operatorname{Mod}_A^{\otimes})$ be a fiber functor inducing monad $T$ and comonad $K$,
$$\xymatrix{
\mathcal{C} \ar@(ul,dl)_{T} \ar[r]_{\omega} &\operatorname{Mod}_A \ar@(ur,dr)^{K} \ar@/_1pc/[l]_{\omega'} \\
}$$ 
If $\omega$ also preserves limits then the following are true:
\begin{enumerate}
\item The tannakization $\operatorname{Aut}^{\otimes}_{\omega}$ is representable by a derived affine group scheme,
\item There is an equivalence of $\infty$-categories $$\operatorname{Rep}_A(\operatorname{Aut}^{\otimes}_{\omega}) \simeq \operatorname{LComod}_K(\operatorname{Mod}_A)$$
\item There is a comparison map $\Phi:\operatorname{Rep}_A(\operatorname{Aut}^{\otimes}_{\omega}) \to \operatorname{Desc}_{\mathcal{C}}(T)$ so that there is a commutative diagram:
$$\xymatrix{
\mathcal{C} \ar[rr]^{Q_T} \ar@/_1pc/[dr]_{\widetilde{\omega}} &&\operatorname{Desc}_{\mathcal{C}}(T)\\
&\operatorname{Rep}_A(\operatorname{Aut}^{\otimes}_{\omega}) \ar[ur]_{\Phi}\\
}$$

\end{enumerate}
\end{theorem}

\begin{cor}
Under the hypothesis of Theorem \ref{theorem} the following are true.
\begin{enumerate}
\item $\omega'$ is monadic $\Longleftrightarrow \Phi$ is an equivalence (The descent category is Tannakian)
\item $\omega$ is comonadic $\Longleftrightarrow \widetilde{\omega}$ is an equivalence (The Tannakization is the motivic Galois group)
\end{enumerate}
\end{cor}

\section{Points of a symmetric monoidal $\infty$-category}

\subsection{Derived schemes} Regarding a scheme as a functor from rings to sets, satisfying certain sheaf conditions with respect to a Grothendieck topology, we want to replace both the source and target by $\infty$-categories. First, we can replace the target by the $\infty$-category of spaces. This broader notion of schemes encompasses the theory of stacks and higher stacks that arise while considering moduli problems in algebraic geometry which exhibit higher automorphisms. Restricting to the the $1$-skeleton of spaces would recover the classical theory of stacks. In order to go from higher stacks to derived stacks we can replace the source category by commutative ring objects in a symmetric monoidal $\infty$-category. There are several notions of derived rings to choose from: dg-algebras, simplicial commutative rings, connective $E_{\infty}$-rings etc. In this paper, our derived commutative rings will be $E_{\infty}$-ring spectra. With this as a starting point a derived scheme should correspond to an  $\infty$-functor $$E_{\infty}\operatorname{-rings} \to \mathcal{S}$$ satisfying certain sheaf conditions with respect to a Grothendieck topology on $E_{\infty}\operatorname{-rings}^{op}$. 

\begin{definition} 
A {\em derived pre-scheme} over $R$ is a left fibration of over $\operatorname{CAlg}_R$. Equivalently, via the Grothendieck construction, it can be expressed as a functor $$X:\operatorname{CAlg}_R \to \mathcal{S}.$$ We will denote by $\mathsf{PreSch}_R$ the $\infty$-category of derived preschemes over $R$. The Grothendieck construction gives an equivalence $\infty$-categories $$\mathsf{PreSch}_R \simeq \mathcal{LF}\operatorname{ib}(\operatorname{CAlg}_R).$$
\end{definition}

\subsubsection{Flat topologies} We define flat topologies {\em fpqc} and {\em fppf} over $\operatorname{CAlg}_S^{op}$. The algebraic notion of faithfully flat module corresponds to the topological notion of a faithful module. The algebraic notion of a module of finite presentation corresponds to the notion of a perfect module. 
\begin{definition}
Let $A$ be an $E_{\infty}$-ring. An $A$-module is $M$ is faithful if for every $A$-module $N$, $M\wedge_AN \simeq * \Rightarrow N\simeq *$.

A set of $A$-algebras $\{A\to B_i\}_{i\in I}$ is a {\em fpqc} cover (or faithful cover) of $A$ if for each $A$-module $N$ with $N\wedge_AB_i \simeq *$ for every $i$, we have $N \simeq *$. In particular, a single faithful $A$-algebra $B$ covers $A$ in this sense. 
\end{definition}

\begin{definition}
A set of $A$-algebras  $\{A \to B_i\}_{i \in I}$ is a {\em fppf} cover of $A$ if it is a faithful cover and every $B_i$ is a perfect $A$-module.
\end{definition}

\begin{definition} Let $\tau$ be a topology on $\operatorname{CAlg}_R^{op}$. A derived prescheme over $R$ is a sheaf in the $\tau$-topology if the the associated functor $X \in \operatorname{Fun}(\operatorname{CAlg}_R,\mathcal{S})$ satisfies the following properties:
\begin{enumerate}
\item If $\{A_i\}$ is a finite family of objects in $\operatorname{CAlg}_R$, then $X(\times_i A_i) \simeq \times_i X(A_i)$
\item Let $f:A \to B$ be a $\tau$-covering and let $C^{\bullet}(B/A)$ be the cobar complex assciated with $f$, then $$ X(\varprojlim C^{\bullet}(B/A)) \simeq \varprojlim X(C^{\bullet}(B/A)).$$
\end{enumerate}

A derived prescheme is a {\em derived scheme} if it is a sheaf for the flat topology.
\end{definition}

Denote by the $\mathsf{Sch}_R \subseteq \mathcal{LF}\operatorname{ib}(\operatorname{CAlg}_R)$ the full $\infty$-subcategory spanned by derived schemes over $R$. For any $A \in \operatorname{CAlg}_R$ we define $\operatorname{Spec}(A)$ to be the functor $\operatorname{CAlg}_R \to \mathcal \mathcal{S}$ co-representable by $A$, this functor is a scheme, $\operatorname{Spec}(R) \in \mathsf{Sch}_R$. We shall call $\operatorname{Spec}(A)$ a {\em derived affine scheme} over $R$. Let $\operatorname{Aff}_R \subseteq \mathsf{Sch}_R$ be the full $\infty$-subcategory spanned by derived affine schemes over $R$. In summary, there are are inclusions of $\infty$-categories $$\operatorname{Aff}_R \subseteq \mathsf{Sch}_R \subseteq \mathcal{LF}\operatorname{ib}(\operatorname{CAlg}_R).$$

\begin{definition}
A derived scheme $X$ is {\em algebraic} if it can be covered by a derived affine scheme $\mathsf{Spec}(A)$ and if it has {\em affine diagonal}. Equivalently, there exists a cosimplicial object $A^{\bullet}$ in $\operatorname{CAlg}_R$, so that $X$ is equivalent to the colimit of the simplicial derived affine scheme $\operatorname{Spec}(A^{\bullet})$ in $\mathcal{LF}\operatorname{ib}(\operatorname{CAlg}_R)$.
\end{definition}

\subsection{The moduli functor} Let $A$ be an $\mathbb{E}_{\infty}$-ring and let $\mathcal{C}^{\otimes}$ be a symmetric monoidal presentable stable $\infty$-category. We construct a $\infty$-functor $$\mathsf{M}_{\mathcal{C}^{\otimes}}:\operatorname{CAlg}_{\mathbb{S}} \to \mathcal{S}$$ associated with any $\mathcal{C}^{\otimes}$ that sends an $\mathbb{E}_{\infty}$-ring $B$ to the space $\operatorname{Fun}^{L,\otimes}(\mathcal{C}^{\otimes},\operatorname{Mod}_B^{\otimes})$ of left adjoint symmetric monoidal $\infty$-functors. 

\subsubsection{} There is a cocartesian fibration $\operatorname{LMod}(\operatorname{Sp}) \to \operatorname{Alg}(\operatorname{Sp})$ (informally for $R\to S \in \operatorname{Alg}(\operatorname{Sp})$ and $(R,M) \in \operatorname{LMod}(\operatorname{Sp})$, $M \to M \otimes_RS$ is the cocartesian edge lying over it). Thus via the Grothendieck construction it gives an $\infty$-functor $\operatorname{Alg}(\operatorname{Sp}) \to \operatorname{Cat}_{\infty}$ which factors through $\operatorname{Alg}(\operatorname{Sp}) \to Pr^{L,\sigma}$. It carries a $\mathbb{E}_{1}$-ring $R$ to  $\operatorname{LMod}_R(\operatorname{Sp})$, the $\infty$-category of left $R$-modules. This functor is extended to a functor between the $\infty$-categories of commutative algebra objects $$QC:\operatorname{CAlg}(\operatorname{Sp}) \to \operatorname{CAlg}(Pr^{L,\sigma})$$ which carries a $\mathbb{E}_{\infty}$-ring $A$ to $\operatorname{Mod}_A(\operatorname{Sp})$, the symmetric monoidal stable $\infty$-category of $A$-modules. We shall denote this category by $\operatorname{Mod}_A^{\otimes}$. More generally there is a $\infty$-functor $$QC:\operatorname{CAlg}_A \to \operatorname{CAlg}_A(Pr^{L,\sigma})$$ sending $A$-algebras to their $A$-linear categories of modules.

There is a co-cartesian fibration $\mathcal{QC} \to \operatorname{CAlg}_A$ classified by $QC$ which is obtained as a pullback of co-cartesian squares:
$$\xymatrix{
\mathcal{QC} \ar[r] \ar[d] &\operatorname{LMod}_A(\operatorname{Sp}) \ar[d] \\
\operatorname{CAlg}_A \ar[r] &\operatorname{Alg}_A(\operatorname{Sp})\\
}$$

Given a derived (pre)scheme $X$ over $A$ encoded as a co-cartesian fibration $X \to \operatorname{CAlg}_A$, we can define the $\infty$-category $QC(X)$ of quasi-coherent sheaves on $X$ to be maps of co-cartesian fibrations 
$$\xymatrix{
X \ar[dr] \ar[r] &\mathcal{QC} \ar[d]\\
&\operatorname{CAlg}_A}$$

$$QC(X) = \operatorname{Maps}_{\operatorname{co}\mathcal{CF}\operatorname{ib}(\operatorname{CAlg}_S)}(X,\mathcal{QC})$$

In general, any derived scheme $X$ can be written as a colimit of affine derived schemes $X \simeq \operatorname{colim}_{U \in \operatorname{Aff}_{/X}} U$. then one defines $QC(X)$ to be the limit, in the $\infty$-category of $\infty$-categories, of the corresponding diagram of $\infty$-categories $$QC(X) = \lim_{U \in \operatorname{Aff}_{/X}} QC(U).$$ The assigment of the category of quasi-coherent sheaves to a derived prescheme is $\infty$-functorial and takes values in the $\infty$-category of symmetric monoidal stable presetable $\infty$-categories. $$QC:\mathsf{PreSch}_S^{op} \to \operatorname{CAlg}(\mathcal{P}r^{L,\sigma})$$ When $X$ is algebraic, by choosing a cover $U \to X$ and the associated simplicial derived affine scheme $U_{\bullet} \to X$ one can describe $QC(X)$ as the totalization of the cosimplicial $\infty$-category $QC(U_{\bullet})$. 

\subsubsection{} Given a small $\infty$-category $\mathcal{C}$, the Yoneda map $Y_{\mathcal{C}}$ is an $\infty$-functor $\operatorname{Cat}_{\infty} \to \mathcal{S}$ sending $\infty$-catgeory $\mathcal{D}$ to the space $\operatorname{Fun}(\mathcal{C},\mathcal{D})$. This is encoded as the left fibration $\left(\operatorname{Cat}_{\infty}\right)_{\mathcal{C}/} \to \operatorname{Cat}_{\infty}$ of simpilcial sets. 

Given a symmetric monoidal stable $\infty$-category $\mathcal{C}^{\otimes}$. Consider the left fibration $$\left(\operatorname{CAlg}(Pr^{L,\sigma})\right)_{\mathcal{C}^{\otimes}/} \to \operatorname{CAlg}(Pr^{L,\sigma}).$$ The associated $\infty$-functor, denoted by $Y_{\mathcal{C}^{\otimes}}: \operatorname{CAlg}(Pr^{L,\sigma}) \to \mathcal{S}$, carries a symmetric monoidal stable $\infty$-category $\mathcal{D}^{\otimes}$ to the space of all left adjoint symmetric monoidal functors $\operatorname{Fun}^{L,\otimes}(\mathcal{C}^{\otimes}, \mathcal{D}^{\otimes})$.

\begin{definition} The derived (pre)scheme $\mathsf{M}_{\mathcal{C}^{\otimes}}$ is obtained as a composition of $\infty$-functors
$$\xymatrix{
\operatorname{CAlg}_{\mathbb{S}} \ar[r]^{QC} \ar@/_2pc/[rr]_{\mathsf{M}_{\mathcal{C}^{\otimes}}} &\operatorname{CAlg}(Pr^{L,\sigma}) \ar[r]^{Y_{\mathcal{C}^{\otimes}}} &\mathcal{S}\\
}$$
\end{definition}

Alternately, it is the $\infty$-functor associated with the left fibration of simplicial sets $\mathsf{M}_{\mathcal{C}^{\otimes}} \to \operatorname{CAlg}_{\mathbb{S}}$ where the following diagram is a pullback square of co-cartesian fibrations.
$$\xymatrix{
\mathsf{M}_{\mathcal{C}^{\otimes}} \ar[r] \ar[d] &\left(\operatorname{CAlg}(Pr^{L,\sigma})\right)_{\mathcal{C}^{\otimes}/} \ar[d]\\
\operatorname{CAlg}_{\mathbb{S}} \ar[r]_{QC}  &\operatorname{CAlg}(Pr^{L,\sigma})\\
}$$

\begin{prop}
For $\mathcal{C}$ a presentable stable symmetric monoidal $\infty$-category, the functor $\mathsf{M}_{\mathcal{C}}^{\otimes}$ is a sheaf in the {\em fppf} topology. 
\end{prop}

\begin{proof} 
We need to show that for a {\em fppf} cover $f:A \to B$ of $E_{\infty}$-rings $\mathsf{M}_{\mathcal{C}^{\otimes}}$ carries the cobar complex associated with $f$ to the limit of the associated cosimplicial diagram of spaces. Recall that the amitsur complex of $f$.

By faithfully dualizable descent if $f:A\to B$ is a {\em fppf}-cover in $\operatorname{CAlg}_S$, then $f$ is of effective descent for modules (see \cite[Thm 1.1]{GD}). Precisely there is an equivalence $$A\operatorname{-mod} \simeq \lim C^{\bullet}(B/A)\operatorname{-mod}$$ of $\infty$-categories in $\operatorname{CAlg}(\mathcal{P}r^{L,\sigma})$. Therefore,

\begin{equation*}
\begin{split}
\mathsf{M}_{\mathcal{C}^{\otimes}}(A) &= \operatorname{Fun}^{L,\otimes}(\mathcal{C},A\operatorname{-mod})\\
&\simeq \operatorname{Hom}_{\operatorname{CAlg}(\mathcal{P}r^L)}(\mathcal{C},\lim C^{\bullet}(B/A)\operatorname{-mod}) \\
&\simeq \lim \operatorname{Hom}_{\operatorname{CAlg}(\mathcal{P}r^L)}(\mathcal{C}, C^{\bullet}(B/A)\operatorname{-mod})\\
&= \lim \mathsf{M}_{\mathcal{C}^{\otimes}}(C^{\bullet}(B/A))\\
\end{split}
\end{equation*} 

\end{proof}
\subsubsection{Pullbacks} Given a map of symmetric monoidal stable $\infty$-categories $\mathcal{C}^{\otimes} \to \mathcal{D}^{\otimes}$, there is a induced map of derived schemes $\mathsf{M}_{\mathcal{D}^{\otimes}} \to \mathsf{M}_{\mathcal{C}^{\otimes}}$. Then the following is a pullback diagram in the $\infty$-category $\mathsf{Sch}_{\mathbb{S}}$.
$$\xymatrix{
\mathsf{M}_{\left(\mathcal{D}\otimes_{\mathcal{C}}\mathcal{D}\right)^{\otimes}} \ar[r] \ar[d] &\mathsf{M}_{\mathcal{D}^{\otimes}} \ar[d]\\
\mathsf{M}_{\mathcal{D}^{\otimes}} \ar[r] &\mathsf{M}_{\mathcal{C}^{\otimes}}\\
}$$

This follows from the pullback square in $\operatorname{Cat}_{\infty}$
$$\xymatrix{
\left(\operatorname{CAlg}(Pr^{L,\sigma})\right)_{\left(\mathcal{D}\otimes_{\mathcal{C}}\mathcal{D}\right)^{\otimes}/} \ar[d] \ar[r] &\left(\operatorname{CAlg}(Pr^{L,\sigma})\right)_{\mathcal{D}^{\otimes}/} \ar[d]\\
\left(\operatorname{CAlg}(Pr^{L,\sigma})\right)_{\mathcal{D}^{\otimes}/} \ar[r] &\left(\operatorname{CAlg}(Pr^{L,\sigma})\right)_{\mathcal{C}^{\otimes}/} \\
}$$

\subsection{Tannakian formalism} There is a symmetric monoidal $\infty$-functor $$\mathsf{M}:\operatorname{CAlg}(Pr^{L,\sigma}) \to \mathsf{PreSch}_{\mathbb{S}}^{op}$$ sending $\mathcal{C}^{\otimes}$ to the derived (pre)scheme $\mathsf{M}_{\mathcal{C}^{\otimes}}$. The functor $\mathsf{M}$ is left adjoint to $QC$.

The construction of the functor $\mathsf{M}$ is related to Tannakian formalism in derived algebraic geometry. The Tannakian formalism attempts to identify the image of $$QC:\mathsf{AlgSch}_{\mathbb{S}}\to \operatorname{CAlg}(Pr^{L,\sigma})$$ (recognition) and possibly recontruct $X$ from $QC(X)^{\otimes}$. 

\begin{enumerate}

\item {\em Recognition.} Given a symmetric monoidal stable $\infty$-category $\mathcal{C}^{\otimes}$, we note that there is a counit map of the adjunction, $QC(\mathsf{M}_{\mathcal{C}^{\otimes}}) \to \mathcal{C}^{\otimes}$. This is always an equivalence. Therefore the recognition problem is related to the question of algebraicity of $\mathsf{M}_{\mathcal{C}^{\otimes}}$. We do not address this question in this generality here. However we can show that under certain finiteness restrictions, $\mathsf{M}_{\mathcal{C}^{\otimes}}$ has affine diagonal.

\item {\em Reconstruction.} Given a derived (pre)scheme $X$, there is a unit of the adjunction, $X \to \mathsf{M}_{QC(X)^{\otimes}}$.  The reconstruction problem is related to the question when is this map an equivalence. We do not address this question here. However we note that for derived affine schemes the reconstruction works. Let $A$ be an $\mathbb{E}_{\infty}$-ring, then $$\mathsf{Spec}\,A \to \mathsf{M}_{\operatorname{Mod}_A^{\otimes}}$$ is an equivalence. This is a consequence of the higher algebra version of the Eilenberg-Watts theorem which states that given $\mathbb{E}_1$-rings $A$ and $B$, there is an equivalence of $\infty$-categories $$\operatorname{Fun}^L(\operatorname{Mod}_A, \operatorname{Mod}_B) \simeq {}_A\operatorname{BiMod}_B(\operatorname{Sp})$$ which under the symmetric monoidal restriction reduces to the equivalence of spaces $$\operatorname{Fun}^{L,\otimes}(\operatorname{Mod}_A^{\otimes}, \operatorname{Mod}_B^{\otimes}) \simeq \operatorname{Map}_{\mathbb{E}_1\operatorname{-rings}}(A,B).$$

\end{enumerate}

\section{The Tannakization group scheme}

\subsection{Derived group schemes}

From the functor of points perspective an ordinary group scheme is a group valued functor in the category of commutative rings so that the underlying set valued functor can be represented by a scheme. A derived group scheme is an $\infty$-functor from the $\infty$-category of $\mathbb{E}_{\infty}$-rings to group objects in $\mathcal{S}$, so that the underlying $\mathcal{S}$-valued functor is a derived scheme. The group objects in $\mathcal{S}$ are group-like $\mathbb{E}_{1}$-spaces.

\begin{definition}
Let $A$ be an $\mathbb{E}_{\infty}$-ring. A {\em derived group scheme} $\mathbb{G}$ over $A$, is a functor 
$$\mathbb{G}:\operatorname{CAlg}_A \to \operatorname{Grp}(\mathcal{S})$$ so that the underlying derived prescheme $\operatorname{CAlg}_A \to \operatorname{Grp}(\mathcal{S}) \to \mathcal{S}$ is representable by a derived scheme $X$. If $X$ is affine, then $G$ is a {\em derived affine group scheme}.

Denote by $\mathsf{GrpSch}_A$ the $\infty$-category of group schemes over $A$.
\end{definition}

\subsubsection{}\label{groupsch}
The $\infty$-category $\operatorname{Grp}(\mathcal{C})$ is the full subcategory of $\operatorname{Fun}(N(\Delta)^{op},\mathcal{S})$ spanned by group objects. A functor $\mathbb{G}:\operatorname{CAlg}_A \to \operatorname{Fun}(N(\Delta)^{op},\mathcal{S})$ is equivalent to a functor $\mathbb{G}':N(\Delta)^{op}\to \operatorname{Fun}(\operatorname{CAlg}_A,\mathcal{S})$. The functor $\mathbb{G}$ factors through $\operatorname{Grp}(\mathcal{S})$ is equivalent to $\mathbb{G}'$ being a group object in $\operatorname{Fun}(\operatorname{CAlg}_A,\mathcal{S})$. There is an equivalence $$\operatorname{Fun}(\operatorname{CAlg}_A,\operatorname{Grp}(\mathcal{S})) \simeq \operatorname{Grp}(\operatorname{Fun}(\operatorname{CAlg}_A,\mathcal{S})).$$ A object in $\operatorname{Grp}(\operatorname{Fun}(\operatorname{CAlg}_A,\mathcal{S}))$ is a derived group scheme if the image under the map $$\operatorname{Grp}(\operatorname{Fun}(\operatorname{CAlg}_A,\mathcal{S})) \to \operatorname{Fun}(\operatorname{CAlg}_R,\mathcal{S})$$ is a derived scheme. Therefore a derived group scheme over $A$ is a group object in the $\infty$-category of derived schemes over $A$. There is an equivalence of categories $$\mathsf{GrpSch}_A \simeq \operatorname{Grp}(\mathsf{Sch}_A).$$

A derived affine group scheme over $A$ is thus an object in $\operatorname{Grp}(\mathsf{Aff}_A)$. An affine group scheme is equivalent to a functor $F:N(\Delta) \to \operatorname{CAlg}_A$ so that $F^{op}:N(\Delta)^{op} \to \mathsf{Aff}_A$ is a group object in $\mathsf{Aff}_A$. Therefore there is a natural equivalence $$\operatorname{CHopf}_A^{op} \simeq \operatorname{Grp}(\mathsf{Aff}_A).$$

\subsubsection{Group scheme actions on $\infty$-categories}

Let $G \in \operatorname{Grp}(\mathcal{S})$. Given a presentable $\infty$-category and an object $X \in \mathcal{C}$. A $G$-action on $X$ is a morphism $G \otimes X \to X$ in $\mathcal{C}$ that satisfies the usual group action axioms upto coherent homotopies. This can be made precise in the following way. 

\begin{definition}
Let $\mathsf{B}G$ be the classifying space of $G$. This is an $\infty$-groupoid. Then a $G$-action on $X$ is a functor of $\infty$-categories $$f:\mathsf{B}G \to \mathcal{C}$$ so that the object in $\mathsf{B}G$ maps to $X \in \mathcal{C}$. 

Alternatively, let $\mathsf{B}\operatorname{Aut}_X(\mathcal{C}) \subseteq \mathcal{C}$ be the full sub $\infty$-groupoid spanned by $X$. Then a $G$ action on $X$ is a map of $\infty$-groupoids $$\mathsf{B}G \to \mathsf{B}\operatorname{Aut}_X(\mathcal{C}).$$

The $\infty$-category $\operatorname{Fun}(\mathsf{B}G,\mathcal{C})$ is the category of $G$-objects in $\mathcal{C}$. 

Define $X^{hG} = \lim(f)$. The simpicial model for $\mathsf{B}G$ gives rise to a {\em group cobar complex} $C^{\bullet}(G;X)$, which is a cosimplicial object in $\mathcal{C}$. The fixed points $X^{hG} \simeq \operatorname{Tot}(C^{\bullet}(G;X))$.
\end{definition}

\begin{remark}({\em Group action on an $\infty$-category})\label{fixedcategory}
The action of a group $G$ on an $\infty$-category $\mathcal{C}$ is given by a functor $$\mathsf{B}G \to \mathsf{B}\operatorname{Aut}_{\mathcal{C}}(\operatorname{Cat}_{\infty}).$$ An object $X$ of $\mathcal{C}$ will be called a {\em $G$-equivariant object of $\mathcal{C}$} if $X$ is an object of $\mathcal{C}^{hG}$.

Informally, the objects of $\mathcal{C}^{hG}$ consist of the following data:
\begin{enumerate}

\item An object $X \in \mathcal{C}$
\item An equivalence $\phi_{g,X}:g.X \to X$ for all $g\in G$
\item A $2$-simplex $\xymatrix{&X\\ g_1.X \ar[ur]^{\phi_{g_1,X}} &&g_2g_1.X \ar[ul]_{\phi_{g_1g_2,X}} \ar[ll]^{g_2,g_1.X}}$ for all $(g_1,g_2) \in G^2$
\item $\cdots$
\end{enumerate}
\end{remark}

\subsubsection{}
Given a derived group scheme $\mathbb{G}:\operatorname{CAlg}_R \to \operatorname{Grp}(\mathcal{S})$, and a presentable $\infty$-category $\mathcal{C}$. A $\mathbb{G}$ action on $X \in \mathcal{C}$ is a family of functors $\mathsf{B}(\mathbb{G}_A) \to \mathsf{B}\operatorname{Aut}_X(\mathcal{C})$ for every $A \in \operatorname{CAlg}_R$ and a for every map $A \to B$ in $\operatorname{CAlg}_R$ diagrams 
$$\xymatrix{
\mathsf{B}(\mathbb{G}_A) \ar[d] \ar[r] &\mathsf{B}\operatorname{Aut}_{X}(\mathcal{C})\\
\mathsf{B}(\mathbb{G}_B) \ar[ur] 
}$$ which commute upto coherent homotopies. The following definition will make this precise.

\begin{definition}
Given a derived group scheme $\mathbb{G}$ over $R$, the {\em classifying stack} is a co-cartesian fibration of $\infty$-categories
$$\xymatrix{
\mathsf{B}\mathbb{G} \ar[d]^p \\
\operatorname{CAlg}_R
}$$
which under the Grothendieck construction corresponds to the functor $\operatorname{CAlg}_R \to \operatorname{Cat}_{\infty}$ that acts on objects by taking $A$ to the $\infty$-category $\mathsf{B}\mathbb{G}_A$. 
\end{definition}

\begin{definition}
Given an object $X$ of an $\infty$-category $\mathcal{C}$, let $p_0:\underline{\mathsf{B}\operatorname{Aut}_X(\mathcal{C})} \to \operatorname{CAlg}_R$ denote the {\em constant} co-cartesian fibration. An action of $\mathbb{G}$ on $X$ is a map of co-cartesian fibrations 
$$\xymatrix{
\mathsf{B}\mathbb{G}\ar[dr]_p \ar[rr] &&\underline{\mathsf{B}\operatorname{Aut}_X(\mathcal{C})} \ar[dl]^{p_0}\\
&\operatorname{CAlg}_R
}$$
\end{definition}

\subsubsection{Category of Representations} Let $\mathbb{G}$ be a derived group scheme over an $\mathbb{E}_{\infty}$-ring $R$. A $\mathbb{G}$-representation over $R$ is informally an $R$-module $N$ with a $\mathbb{G}$-action. By earlier discussion this is encoded by a coherent family of $\infty$-functors $$B(\mathbb{G}_A) \to B\operatorname{Aut}_N(\operatorname{Mod}_R)$$ for $A \in \operatorname{CAlg}_R$. This is expressed as a map of co-cartesian fibrations $\phi:B\mathbb{G} \to \underline{B\operatorname{Aut}_N(\operatorname{Mod}_R)}$. 

\begin{definition}
The category of $\mathbb{G}$-representations in $R$,
$$\operatorname{Rep}_R(\mathbb{G}) := \left( \operatorname{Mod}_R \right)^{h\mathbb{G}}  = \varprojlim \phi.$$
\end{definition}

There is an explicit formula for computing this. The derived group scheme $\mathbb{G}$ has an underlying scheme $\mathbb{G}$. The group structure is encoded as a simplicial object $\mathbb{G}_{\bullet}\in \operatorname{Fun}(N(\Delta^{op}),\mathsf{Sch}_R)$. Then, 
$$\operatorname{Rep}_R(\mathbb{G}) \simeq \varprojlim QC(\mathbb{G}_{\bullet}).$$

If $\mathbb{G}$ is a derived {\em affine} group scheme $\mathsf{Spec}\, \mathcal{H}$, where $\mathcal{H}$ is a Hopf algbebra over $R$, the category of $\mathbb{G}$-representations over $R$, $$\operatorname{Rep}_R(\mathbb{G}) \simeq \operatorname{LComod}_{\mathcal{H}}(\operatorname{Mod}_R).$$

\subsection{Pointed derived schemes and loop schemes}
\begin{definition} A derived (pre)scheme $X$ over $R$ is said to be pointed if when considered as an $\infty$-functor $X:\operatorname{CAlg}_R \to \mathcal{S}$, $X$ lifts to an $\infty$-functor taking values in $\mathcal{S}_*$, the $\infty$-category of pointed spaces: the following commutative diagram exists in $\operatorname{Cat}_{\infty}$.
$$\xymatrix{
&\mathcal{S}_* \ar[d]\\
\operatorname{CAlg}_R \ar[r]_{X} \ar@{.>}[ur] &\mathcal{S}\\
}$$
Equivalently, the structure map $X \to \mathsf{Spec}\,R$ admits a section.
\end{definition}

\subsubsection{}To any pointed derived scheme $X$ over $R$, we can associate a derived group scheme $\Omega X$ over $R$, defined as the composition of $\infty$-functors 
$$\xymatrix{
\operatorname{CAlg}_R \ar[r]^X \ar@/_1pc/[rr]_{\Omega X} &\mathcal{S}_* \ar[r]^{\Omega_*} &\operatorname{Grp}(\mathcal{S})\\
}$$
where $\Omega_*:\mathcal{S}_* \to \operatorname{Grp}(\mathcal{S})$ is the $\infty$-functor sending a pointed space $i:\Delta^0 \to X$ to its space of based loops $\Omega_*X = \operatorname{\check{C}ech}(\Delta^0\to X)$.

Being a pointed derived scheme, $X$ admits a section $\mathsf{Spec}\,R \to X$. The derived loop scheme $\Omega X$ is the derived group scheme over $\mathsf{Spec}\,R$ arising from the \v{C}ech nerve of the map $\mathsf{Spec}\,R \to X$.
$$\Omega X \simeq \operatorname{\check{C}ech}(\mathsf{Spec}\,R \to X)$$ The underlying derived scheme is the pullback $\mathsf{Spec}\,R \times_{X} \mathsf{Spec}\,R.$

\begin{remark}
Let $A$ be an $\mathbb{E}_{\infty}$-algebra over $R$ such that $X \times_{\mathsf{Spec}\,R} \mathsf{Spec}\,A$ is pointed. Then the projection map $X\times_{\mathsf{Spec}\,R}\mathsf{Spec}\,A \to \mathsf{Spec}\,A$ has a section $\mathsf{Spec}\,A \to X \times_{\mathsf{Spec}\,R} \mathsf{Spec}\,A$ (which is identity in the second component). 

The loop scheme $$\Omega(X \times_{\mathsf{Spec}\,R} \mathsf{Spec}\,A) = \operatorname{\check{C}ech}\left(\mathsf{Spec}\,A \to X \times_{\mathsf{Spec}\,R} \mathsf{Spec}\,A\right)$$ has underlying derived scheme is  $\mathsf{Spec}\,A \underset{X \times_{\mathsf{Spec}\,R} \mathsf{Spec}\,A}{\times} \mathsf{Spec}\,A$ over $\mathsf{Spec}\,A$. However since the section is identity on $\mathsf{Spec}\,A$, there is an equivalence $$\mathsf{Spec}\,A \underset{X \times_{\mathsf{Spec}\,R} \mathsf{Spec}\,A}{\times} \mathsf{Spec}\,A \simeq \mathsf{Spec}\,A \times_X\mathsf{Spec}\,A$$ of derived schemes over $\mathsf{Spec}\,A$ and the loop scheme $$\Omega(X \times_{\mathsf{Spec}\,R}\mathsf{Spec}\,A) \simeq \operatorname{\check{C}ech}\left(\mathsf{Spec}\,A \to X\right).$$

\end{remark}

\subsection{Construction of $\operatorname{Aut}^{\otimes}_{\omega}$}
Let $A$ be an $\mathbb{E}_{\infty}$-ring and let $\mathcal{C}^{\otimes}$ be a $A$-linear symmetric monoidal presentable stable $\infty$-category. Given a fiber functor $\omega \in \operatorname{Fun}^{L,\otimes}(\mathcal{C}^{\otimes}, \operatorname{Mod}_A)$, we construct a derived group scheme $\operatorname{Aut}^{\otimes}_{\omega}$ over $\mathsf{Spec}\,A$. This is expressed as an $\infty$-functor $$\operatorname{Aut}^{\otimes}_{\omega}: \operatorname{CAlg}_A \to \operatorname{Grp}(\mathcal{S})$$ that sends a commutative $A$-algebra $B$ to the group object $\Omega_*\operatorname{Fun}^{L,\otimes}(\mathcal{C}^{\otimes},\operatorname{Mod}_B^{\otimes})$ of loops based at the map $\mathcal{C}^{\otimes} \to \operatorname{Mod}_A^{\otimes} \to \operatorname{Mod}_B^{\otimes}$, the first map is $\omega$, and the second map is induced by the $A$-algebra structure map $A\to B$ on $B$. 

\begin{definition}
Given a $\mathcal{C}^{\otimes}$ with a fiber functor $\omega$, the derived scheme $\mathsf{M}_{\mathcal{C}^{\otimes}}\times_{\mathsf{Spec}\,\mathbb{S}} \mathsf{Spec}\,A$ is pointed, and therefore lifts to a functor $\operatorname{CAlg}_A \to \mathcal{S}_*$. Define $\operatorname{Aut}^{\otimes}_{\omega}$ to be the associated derived loop scheme $$\operatorname{Aut}^{\otimes}_{\omega} = \Omega(\mathsf{M}_{\mathcal{C}^{\otimes}}\times_{\mathsf{Spec}\,\mathbb{S}}\mathsf{Spec}\,A).$$
\end{definition}

\begin{equation}
\xymatrix{
\operatorname{CAlg}_A \ar@/^2pc/[rr]^{\mathsf{M}_{\mathcal{C}^{\otimes}}\times_{\mathsf{Spec}\,\mathbb{S}}\mathsf{Spec}\,A} \ar[r]^{QC} \ar[drr] \ar@/_1pc/[ddrrr]_{\operatorname{Aut}^{\otimes}_{\omega}} &\operatorname{CAlg}(Pr^{L,\sigma}) \ar[r]^{Y_{\mathcal{C}^{\otimes}}} &\mathcal{S} \\
&&\mathcal{S}_* \ar[u]^{\operatorname{forget}} \ar[dr]^{\Omega_*}\\
&&& \operatorname{Grp}(\mathcal{S})\\
}
\end{equation}

\subsubsection{} The fiber functor $\omega$ corresponds to a map of derived schemes $\mathsf{Spec}\,A \to \mathsf{M}_{\mathcal{C}^{\otimes}}$ and the automorphism group scheme is the derived group scheme arising from the \v{C}ech nerve associated with this map.
$$\operatorname{Aut}^{\otimes}_{\omega} \simeq \operatorname{\check{C}ech}(\mathsf{Spec}\,A \to \mathsf{M}_{\mathcal{C}^{\otimes}})$$
The underlying derived scheme is the pullback $\mathsf{Spec}\,A \times_{\mathsf{M}_{\mathcal{C}^{\otimes}}} \mathsf{Spec}\,A.$ 

The cobar construction associated with $\omega$ is a (co-augmented) cosimplicial symmetric monoidal stable $\infty$-category 
$$\xymatrix{
\operatorname{Cobar}(\omega) = \operatorname{Mod}_A^{\otimes} \ar[r] \ar@<1ex>[r] &\left( \operatorname{Mod}_A\otimes_{\mathcal{C}} \operatorname{Mod}_A \right) ^{\otimes} \ar[r] \ar@<1ex>[r] \ar@<-1ex>[r] &\left(\operatorname{Mod}_A \otimes_{\mathcal{C}} \operatorname{Mod}_A\otimes_{\mathcal{C}}\operatorname{Mod}_A \right)^{\otimes} \ar[r] \ar@<1ex>[r] \ar@<-1ex>[r] \ar@<2ex>[r] &\cdots\\
}$$ The pullback property of $\mathsf{M}$ means this is equivalent as a derived scheme to $\mathsf{M}_{\left(\operatorname{Mod}_A\otimes_{\mathcal{C}}\operatorname{Mod}_A\right)^{\otimes}}$. The group structure is encoded in the simplicial derived scheme $$\mathsf{M}_{\operatorname{Cobar}(\omega)}.$$

\subsubsection{Representations of the Tannakization}\label{repcobar}
\begin{equation}
\begin{split}
\operatorname{Rep}_A(\operatorname{Aut}_{\omega}^{\otimes}) &\simeq \varprojlim QC(\operatorname{\check{C}ech}(\mathsf{Spec}\,A \to \mathsf{M}_{\mathcal{C}^{\otimes}}))\\
&\simeq \varprojlim \operatorname{Cobar}(\omega)\\
\end{split}
\end{equation}

\subsubsection{}
Let $f:A\to B$ be a map of $\mathbb{E}_{\infty}$-rings and let $f^*:\operatorname{Mod}_A^{\otimes} \to \operatorname{Mod}_B^{\otimes}$ be the associated pullback symmetric monoidal functor. This is a fiber functor on $\operatorname{Mod}_A^{\otimes}$ the tannakization of which is the derived affine group scheme $\mathsf{Spec}\,(B\otimes_AB)$ over $\mathsf{Spec}\,B$. The group structure arises from the \v{C}ech nerve associated to $\mathsf{Spec}\,B \to \mathsf{Spec}\,A$.

Slightly more generally, let $Y$ be a derived scheme over $R$ and let $\mathsf{Spec}\,R \to Y$ be a section of the structure map. Then the associated pullback $QC(Y)^{\otimes} \to \operatorname{Mod}_R^{\otimes}$ is a fiber functor. The tannakization of this is equivalent to the derived affine group scheme over $\mathsf{Spec}\,R$ arising from the \v{C}ech nerve of $\mathsf{Spec}\,R \to Y$ (see \cite{IwanariBar}).

\subsection{Fiber functor as a bimodule}

\subsubsection{} By the Eilenberg-Watts theorem for $\mathbb{E}_1$-rings $R$ and $S$, any fiber functor $\operatorname{Mod}_R^{\otimes} \to \operatorname{Mod}_S^{\otimes}$ on $\operatorname{Mod}_R^{\otimes}$ is obtained by tensoring with a $R\operatorname{-}S$-bimodule. The $\infty$-functor that sends a bimodule $M$ to $-\otimes_SM$ gives an equivalence of $\infty$-categories $$\xymatrix{\operatorname{RMod}_S(\operatorname{LMod}_R^{op}) = {}_R\operatorname{BiMod}_S \ar[r]^{\simeq} &\operatorname{Fun}^L(\operatorname{LMod}_R,\operatorname{RMod}_S)}$$

We give a similar characterization of more general fiber functors. Let $\mathcal{C}^{\otimes}$ be a symmetric monoidal stable $\infty$-category and $A$ be a $\mathbb{E}_{\infty}$-ring. 

\begin{prop}
There is an $\infty$-functor $$\operatorname{Mod}_A(\mathcal{C}^{op}) \to \operatorname{Fun}^{R}(\mathcal{C},\operatorname{Mod}_A)$$ sending a $A$-module $M$ in $\mathcal{C}$ to $\operatorname{Hom}_{\mathcal{C}}(M,-)$ which is an equivalence of $\infty$-categories. Restricting to the sub $\infty$-category of compact objects induces and equivalence of $\infty$-categories
$$\operatorname{Mod}_A(\mathcal{C}^{op})^c \simeq \operatorname{Fun}^{L,R}(\mathcal{C},\operatorname{Mod}_A).$$
\end{prop}

\begin{proof}
Apply \cite[Prop.4.8.1.17]{HA}, \cite[Thm.4.8.4.6]{HA}
\end{proof}

Therefore any fiber functor on $\mathcal{C}^{\otimes} \to \operatorname{Mod}_A^{\otimes}$ that preserves limits is equivalent to $\operatorname{Hom}_{\mathcal{C}}(X,-)$ for some $X \in \operatorname{Mod}_A(\mathcal{C}^{op})$. Also, given a commutative $A$-algebra $A \to B$, the composite $\mathcal{C}^{\otimes} \to \operatorname{Mod}_A^{\otimes} \to \operatorname{Mod}_B^{\otimes}$ is the functor $\operatorname{Hom}_{\mathcal{C}}(X\otimes_AB,-)$. 

\begin{definition}
A fiber functor $\omega:\mathcal{C}^{\otimes} \to \operatorname{Mod}_A^{\otimes}$ is said to be {\em compact} if it preserves limits.
\end{definition}

\subsubsection{} Using the identification of limit preserving fiber functors with compact module objects we get an alternate description of the Tannakization group scheme.

\begin{theorem}\label{affine}
Let $\omega:\mathcal{C}^{\otimes} \to \operatorname{Mod}_A^{\otimes}$ be a compact fiber functor. Then the tannakization functor is representable by a derived affine group scheme over $A$.
\end{theorem}

\begin{proof}
Let us use the notation $\operatorname{Map}_A(X,Y)$ for $\operatorname{Hom}_{\operatorname{Mod}_A(\mathcal{C}^{op})^c}(X,Y)$ and $\operatorname{Iso}_A(X,Y)$ for $\operatorname{Iso}_{\operatorname{Mod}_A(\mathcal{C}^{op})^c}(X,Y)$. Here, for an $\infty$-category $\mathcal{D}$, $\operatorname{Iso}_{\mathcal{D}}(-,-) \subset \operatorname{Hom}_{\mathcal{D}}(-,-)$ is the full subcategory spanned by morphisms which become isomorphisms in $h\mathcal{D}$.
\begin{equation}
\begin{split}
\operatorname{Aut}^{\otimes}_{\omega}(B) &= \operatorname{Iso}_B(\omega\otimes_AB,\omega\otimes_AB)\\
&\simeq \operatorname{Iso}_A(\omega,\omega\otimes_AB)\\
&\simeq \left(\operatorname{Iso}_A(\omega,\omega)\right)\otimes_AB\\
&\simeq \operatorname{Hom}_{\operatorname{Mod}_A}(\left(\operatorname{Iso}_A(\omega,\omega)\right)^{\vee}, B)\\
&\simeq \operatorname{Hom}_{\operatorname{CAlg}_A}(\operatorname{Symm}^*\left(\left(\operatorname{Iso}_A(\omega,\omega)\right)^{\vee}\right), B)\\
\end{split}
\end{equation}

The first equivalence follows from a base change left adjoint $-\otimes_AB:\operatorname{Mod}_A(\mathcal{C}) \to \operatorname{Mod}_B(\mathcal{C})$. The stable $\infty$-category $\operatorname{Mod}_A$ is generated by colimits, $\otimes_A$ preserves colimits and $\omega$ is compact. This gives the second equivalence. Here $\operatorname{Symm}^*(N)$ is the free $\mathbb{E}_{\infty}$-algebra generated by $N$.

Therefore $\infty$-functor $\operatorname{Aut}_{\omega}^{\otimes}:\operatorname{CAlg}_A\to \mathcal{S}$ is representable by the $A$-algebra $$\mathcal{H} = \operatorname{Symm}^*\left(\operatorname{Iso}_A(\omega,\omega)^{\vee}\right).$$ Since the underlying derived scheme is a derived group scheme the commutative $A$-algebra $\mathcal{H}$ is a commutative $A$-Hopf algebra.
\end{proof}

\subsubsection{}
The tensor product of stable $\infty$-categories is a category of modules.
\begin{equation}
\begin{split}
\left(\operatorname{Mod}_A\otimes_{\mathcal{C}}\operatorname{Mod}_A\right)^{\otimes} &= QC(\operatorname{Spec}\,A\times_{\mathsf{M}_{\mathcal{C}^{\otimes}}}\operatorname{Spec}\,A)\\
&= QC(\operatorname{Aut}^{\otimes}_{\omega})^{\otimes}\\
&\simeq \operatorname{Mod}_{\mathcal{H}}^{\otimes}\\
\end{split}
\end{equation}
 
This shows $\mathsf{M}_{\mathcal{C}^{\otimes}}$ has $QC$-affine diagonal for compact maps. Let $\operatorname{Spec}\,A \to \mathsf{M}_{\mathcal{C}^{\otimes}}$ be the map defined by a symmetric monoidal $\infty$-functor $\mathcal{C}^{\otimes} \to \operatorname{Mod}_A^{\otimes}$ preserving limits and colimits. Then,
\begin{equation}
\begin{split}
QC(\operatorname{Spec}\,A \times_{\mathsf{M}_{\mathcal{C}^{\otimes}}} \operatorname{Spec}\,B)^{\otimes} &= \left(\operatorname{Mod}_A\otimes_{\mathcal{C}}\operatorname{Mod}_B\right)^{\otimes}\\
&\simeq \left(\operatorname{Mod}_A\otimes_{\mathcal{C}}\operatorname{Mod}_A\right)^{\otimes}\otimes_A\operatorname{Mod}_B^{\otimes}\\
&\simeq \left(\operatorname{Mod}_{\mathcal{H}\otimes_AB}\right)^{\otimes}\\
\end{split}
\end{equation}

\section{Monadic descent and Tannakization} 

\subsection{The descent category of a monad}
Let $\mathcal{C}$ be an $\infty$-category and $T\in \operatorname{Alg}(\operatorname{Fun}(\mathcal{C},\mathcal{C}))$ be a monad on $\mathcal{C}$. One defines the {\em descent category $\operatorname{Desc}_{\mathcal{C}}(T)$ for $T$ in $\mathcal{C}$} as a category fo comodules in the category of $T$-modules $\operatorname{LMod}_T(\mathcal{C})$. One should think of objects in $\operatorname{Desc}_{\mathcal{C}}(T)$ as $T$-modules $+$ ``descent data". The definition goes as follows:

\begin{equation}
\xymatrix{
&&\operatorname{LComod}_{K_T}(\operatorname{LMod}_T(\mathcal{C})) \\
\mathcal{C} \ar@(ul,dl)_{T} \ar[dr]^{F_T}  \ar[rr]^{Q_T} &&\operatorname{Desc}_{\mathcal{C}}(T) \ar@{=}[u]^{\operatorname{def}}\ar[dl]_{U_{K_T}}\\ 
&\operatorname{LMod}_T(\mathcal{C}) \ar@/_1pc/[ur]_{F_{K_T}} \ar@/^1pc/[ul]^{U_T} \ar@(dl,dr)_{K_T:=F_T\circ U_T}\\
}
\end{equation}

We unpack the above diagram. First, $T$ is a monad on the $\infty$-category $\mathcal{C}$. The $\infty$-category $\mathcal{C}$ is left-tensored over $\operatorname{Fun}(\mathcal{C},\mathcal{C})$ and $\operatorname{LMod}_{T}(\mathcal{C})$ is the $\infty$-category of left $T$-modules in $\mathcal{C}$. It is related to the original category $\mathcal{C}$ by the Eilenberg-Moore adjunction $F_T:\mathcal{C} \leftrightarrows \operatorname{LMod}_T(\mathcal{C}):U_T$, where $F_T$ is the free $T$-module functor and the $U_T$ is the forgetful functor. This functor realizes $T$ in the sense that $U_T\circ F_T = T$. The other composition $K_T = F_T\circ U_T$ defines a comonad on $\operatorname{LMod}_T(\mathcal{C})$. The descent category is defined to be the $\infty$-category of left comodules in $\operatorname{LMod}_T(\mathcal{C})$ with respect to the comonad $K_T$. 

The $\infty$-category $\operatorname{Desc}_{\mathcal{C}}(T) = \operatorname{LComod}_{K_T}(\operatorname{LMod}_T(\mathcal{C}))$ is again related to the original category $\operatorname{LMod}_T(\mathcal{C})$ by a co-Eilenberg-Moore adjunction $F_{K_T}:\operatorname{LMod}_T(\mathcal{C}) \rightleftarrows \operatorname{Desc}_{\mathcal{C}}(T): U_{K_T}$. This realizes the comonad $K_T$ in that $U_{K_T}\circ F_{K_T} = K_T$. This induces a functor $$Q_T: \mathcal{C} \to \operatorname{Desc}_{\mathcal{C}}(T).$$ The monad $T$ is a said to satisfy {\em effective descent in $\mathcal{C}$} when the functor $Q_T$ is an equivalence.

\subsubsection{} In the situation where the monad $T$ on $\mathcal{C}$ is induced by an adjunction, there is a factorization of the map $Q_T$. 
Let $\xymatrix{\mathcal{C} \ar[r]_F \ar@(ul,dl)_{T} &\mathcal{D} \ar@/_/[l]_G \ar@(ur,dr)^{K}}$ be an adjunction of $\infty$-categories inducing a monad $T = G\circ F$ on $\mathcal{C}$ and comonad on $K = F\circ G$ on $\mathcal{D}$. The descent category for $T$ on $\mathcal{C}$ can be defined as before. Since the adjunctions $F_T:\mathcal{C} \leftrightarrows \operatorname{LMod}_T(\mathcal{C}):U_T$ and $F:\mathcal{C} \rightleftarrows \mathcal{D}: G$ both realize the monad $T$ on $\mathcal{C}$ there is an induced map $\operatorname{Can}_T:\mathcal{D} \to \operatorname{LMod}_T(\mathcal{C})$ such that $U_T\circ \operatorname{Can}_T = G$ and $\operatorname{Can}_T\circ F = F_T$. 

\begin{equation}
\xymatrix{
\mathcal{C} \ar[r]_F \ar@(ul,dl)_{T} \ar[d]_{F_T} \ar@/^5pc/[drr]^{\operatorname{Can}_K}  \ar[ddr]^{Q_T} &\mathcal{D} \ar@/_1pc/[dr]_{F_K} \ar[dl]^{\operatorname{Can}_T} \ar@/_1pc/[l]_G \ar@(ur,dr)^{K}\\
\operatorname{LMod}_T(\mathcal{C}) \ar@/_1pc/[dr] \ar@/_/[u]_{U_T} \ar@(dl,dr)_{K_T} &&\operatorname{LComod}_K(\mathcal{D}) \ar[ul]_{U_K} \ar@{-->}[dl]^{\Phi}\\
&\operatorname{Desc}_{\mathcal{C}}(T) \ar[ul]  \\
}
\end{equation}

Therefore the map $\operatorname{Can}_K:\mathcal{D} \to \operatorname{LMod}_T(\mathcal{C})$ commutes with the comonad actions of $K$ and $K_T$, i.e. $\operatorname{Can}_T\circ K \simeq K_T\circ \operatorname{Can}_T$. Hence if $X \in \mathcal{D}$ is a comodule over $K$, the comodule structure being encoded as a cosimplcial object 
$$\xymatrix{
X \ar[r] \ar@<1ex>[r] &K(X) \ar[r] \ar@<1ex>[r] \ar@<-1ex>[r] &K^2(X) \ldots \\
}$$
after applying $\operatorname{Can}_T$, we get a cosimplicial object in $\operatorname{LMod}_T(\mathcal{C})$
$$\xymatrix{
\operatorname{Can}_T(X) \ar[r] \ar@<1ex>[r] &\operatorname{Can}_T(K(X)) \ar[r] \ar@<1ex>[r] \ar@<-1ex>[r] &\operatorname{Can}_T(K^2(X)) \ldots \\
}$$ which is equuivalent as a cosimplicial object in $\operatorname{LMod}_T(\mathcal{C})$ to 
$$\xymatrix{
\operatorname{Can}_T(X) \ar[r] \ar@<1ex>[r] &K_T(\operatorname{Can}_T(X)) \ar[r] \ar@<1ex>[r] \ar@<-1ex>[r] &K_T^2(\operatorname{Can}_T(X)) \ldots \\
}$$ This puts a $K_T$-comodule structure on $\operatorname{Can}_T(X)$. We have arrived at the following proposition.

\begin{prop}\label{comoduletodescentmap}
There is an induced functor $$\Phi:\operatorname{LComod}_K(\mathcal{D}) \to \operatorname{Desc}_{\mathcal{C}}(T).$$ This map is an equivalence if and only if $\operatorname{Can}_T$ is an equivalence, which is the case when $G$ is monadic.

\end{prop}

There is a commutative diagram 
$$\xymatrix{
\mathcal{C} \ar[rr]^{Q_T} \ar[dr]_{\operatorname{Can}_K} &&\operatorname{Desc}_T(\mathcal{C})\\
&\operatorname{LComod}_K(\mathcal{D}) \ar[ur]_{\Phi}\\
}$$

If follows that $T$ is of effective descent in $\mathcal{C}$ if and only of $G$ is monadic and $F$ is comonadic.

\subsection{The Beck Chevalley condition}Given a symmetric monoidal stable $\infty$-category $\mathcal{C}^{\otimes}$ and a fiber functor $\omega:\mathcal{C}^{\otimes} \to \operatorname{Mod}_A^{\otimes}$, let $T$ and $K$ be the associated monad and comonad
$$\xymatrix{
\mathcal{C}^{\otimes}\ar[r]_{\omega} \ar@(ul,dl)_T &\operatorname{Mod}_A \ar@/_1pc/[l]_{\omega'} \ar@(ur,dr)^K\\
}$$ We show that under the assumption that $\omega$ is compact, the forgetful functor $$\operatorname{Rep}_A(\operatorname{Aut}^{\otimes}_{\omega}) \to \operatorname{Mod}_A$$ has a right adjoint and the induced comonad on $\operatorname{Mod}_A$ is equivalent to $K$. Therefore there is a map of $\infty$-categories $$\operatorname{Rep}_A(\operatorname{Aut}^{\otimes}_{\omega}) \to \operatorname{LComod}_K(\operatorname{Mod}_A)$$ which is also an equivalence.

\begin{definition}(Right-adjointability)\cite[4.7.5.13]{HA}\label{rightadjointable}
Given a diagram of $\infty$-categories $\sigma$:

$$\xymatrix{
\mathcal{C} \ar[r]^G  \ar[d]_U &\mathcal{D} \ar@/_1pc/[l]_{H} \ar[d]^V\\
\mathcal{C}' \ar[r] ^{G'}  &\mathcal{D}'  \ar@/_1pc/[l]_{H'} \\
}$$
and a specified equivalence $\alpha: G'\circ U \simeq V\circ G$. We say $\sigma$ is {\em right adjointable} if $G$ and $G'$ admit right adjoints $H$ and $H'$, and the composition transformation $$\xymatrix{ U\circ H \to H\circ G' \circ U \circ H \ar[r]^{\alpha} &H'\circ V \circ G \circ H \to H'\circ V}$$ is an equivalence.
\end{definition}

\begin{prop}\label{bc}
Let $\omega:\mathcal{C}^{\otimes} \to \operatorname{Mod}_A^{\otimes}$ be a compact fiber functor. The commutative square 
$$\xymatrix{
\mathcal{C} \ar[r]^{\omega} \ar[d]_{\omega} &\operatorname{Mod}_A \ar@/_1pc/[l]_{\omega '}\ar[d]^{G}\\
\operatorname{Mod}_A \ar[r]^G &\operatorname{Mod}_A\otimes_{\mathcal{C}}\operatorname{Mod}_A \ar@/_1pc/[l]_{G'}\\
}$$
is right-adjointable. Here $\omega'$ and $G'$ are right adjoint to $\omega$ and $G$ respectively.
\end{prop}
\begin{proof}
By Thm.\ref{affine} there is a commutative Hopf algebra $\mathcal{H}$ over $A$ and symmetric monoidal equivalences 
\begin{enumerate}
\item $\operatorname{Mod}_A\otimes_{\mathcal{C}}\operatorname{Mod}_A \simeq \operatorname{Mod}_{\mathcal{H}}^{\otimes}$
\item $\operatorname{Mod}_A\otimes_{\mathcal{C}}\operatorname{Mod}_B \simeq \left(\operatorname{Mod}_{\mathcal{H}\otimes_AB}\right)^{\otimes}$
\end{enumerate}

Under these equivalences the pushout square of the proposition can be identified with 
$$\xymatrix{
\mathcal{C} \ar[r]^{\omega} \ar[d]_{\omega} &\operatorname{Mod}_A \ar@/_1pc/[l]_{\omega '}\ar[d]^{G}\\
\operatorname{Mod}_A \ar[r]^G &\operatorname{Mod}_{\mathcal{H}} \ar@/_1pc/[l]_{G'}\\
}$$ where $G$ is base change along the unit map $A\to \mathcal{H}$ of the Hopf algebra $\mathcal{H}$, and $G'$ the forgetful functor.

First we define a map $\omega_+:\operatorname{Mod}_A^{\otimes} \to \mathcal{C}^{\otimes}$ by requiring it to satisfy the right adjointability condition. Then we show that $\omega_+$ is the right adjoint to $\omega$. Since $\mathcal{C} \simeq \varprojlim_{\mathcal{C}/}\operatorname{Mod}_B$, an $\infty$-functor $\operatorname{Mod}_A \to \mathcal{C}$ is equivalent to a family of $\infty$-functors $\operatorname{Mod}_A \to \operatorname{Mod}_B$ for every $f:\mathcal{C} \to \operatorname{Mod}_B$ which are compatible upto higher coherent homotopies. 
$$\xymatrix{
\mathcal{C} \ar[r]^{\omega} \ar[d]_{f} &\operatorname{Mod}_A \ar@/_1pc/@{.>}[l]_{\omega_+}\ar[d]^{F} \ar[dl]_{\omega_+(f)}\\
\operatorname{Mod}_B \ar[r]^H   &\operatorname{Mod}_{\mathcal{H}\otimes_AB} \ar@/^1pc/[l]^{H'}\\
}$$
We define $\omega_+:\operatorname{Mod}_A \to \mathcal{C}$ by defining $\omega_+(f) = H'\circ F$ for every $f:\mathcal{C} \to \operatorname{Mod}_B$. We have to check this is a well-defined map. That is, for any map of $\mathbb{E}_{\infty}$-rings $B\to C$, we show there is a canonical equivalence between $g^*\circ\omega_+(f)$ and $\omega_+(g^*\circ f)$.

$$\xymatrix{
\mathcal{C} \ar[r]^{\omega} \ar[d]_{f} &\operatorname{Mod}_A \ar@/_1pc/@{.>}[l]_{\omega_+}\ar[d]^{F} \ar[dl]_{\omega_+(f)} \ar@/_7pc/[ddl]^{\omega_+(g^*\circ f)}\\
\operatorname{Mod}_B \ar[r]^H  \ar[d]_{g^*} &\operatorname{Mod}_{\mathcal{H}\otimes_AB} \ar@/^1pc/[l]^{H'} \ar[d]^{\mathcal{H}\otimes_Ag^*} \\
\operatorname{Mod}_C \ar[r]^K &\operatorname{Mod}_{\mathcal{H}\otimes_AC} \ar@/^1pc/[l]^{K'}\\
}$$

By our definition, $$g^*\circ\omega_+(f) = g^*\circ H' \circ F$$ $$\omega_+(g^*\circ f) = K'\circ (\mathcal{H}\otimes_Ag^*)\circ F.$$ Therefore $g^*\circ\omega_+(f)$ and $\omega_+(g^*\circ f)$ are canonically equivalent if there is  canonical equivalence $$g^*\circ H'\simeq K'\circ (\mathcal{H}\otimes_Ag^*).$$ This follows from the right adjointability of the lower square which is $\operatorname{Mod}_{(-)}$ applied to the pushout square of $\mathbb{E}_{\infty}$-rings 
$$\xymatrix{ 
B \ar[r] \ar[d] &\mathcal{H}\otimes_AB \ar[d] \\ 
C \ar[r] &\mathcal{H}\otimes_AC\\
}$$

Now we show $\omega_+ \simeq \omega'$. Any $X\in \mathcal{C}$ is equivalent to $\lim_{\mathcal{C}/\mathsf{Aff}}f'fM$, where the limit is taken over diagrams $\xymatrix{\mathcal{C} \ar[r]_f &\operatorname{Mod}_B \ar@/_1pc/[l]_{f'}}$ in $Pr^{L,\sigma}$ for $\operatorname{Mod}_B$ is in $\mathsf{Aff}$, the image under $QC$ of $\operatorname{CAlg}_A$. 
 
\begin{equation}
\begin{split}
\omega_+(M) &= \lim_{\mathcal{C}/\mathsf{Aff}} f'f(\omega_+M)\\
&\simeq\lim_{\mathcal{C}/\mathsf{Aff}}(f'H'FM)\\
&\simeq\lim_{\mathcal{C}/\mathsf{Aff}}(\omega'F'FM)\\
&\simeq \omega'\lim_{A/}(F'FM)\\
&\simeq \omega'M\\
\end{split}
\end{equation}

\end{proof}

\subsubsection{} The left adjoint fiber functor $\omega: \mathcal{C} \to \operatorname{Mod}_A$ induces a comonad $K= \omega'\circ \omega$ on $\operatorname{Mod}_A$.
$$\xymatrix{
\mathcal{C} \ar[r]_{\omega} &\operatorname{Mod}_A \ar@/_1pc/[l]_{\omega'} \ar@(ur,dr)^K\\
}$$

\begin{prop}\label{equivcomonads}
Let $\omega$ be compact. The projection map $U:\varprojlim(\operatorname{Cobar}(\omega)) \to \operatorname{Mod}_A$ is has a right adjoint $U^!$ which produces a $\infty$-comonad $K'=U\circ U^!$ on $\operatorname{Mod}_A$. 

$$\xymatrix{
\varprojlim\operatorname{Cobar}(\omega) \ar[r]_{U} &\operatorname{Mod}_A \ar@/_1pc/[l]_{U^!} \ar@(ur,dr)^{K'}\\
}$$

The map $U$ is comonadic and the $\infty$-comonads $K$ and $K'$ are equivalent, $K\simeq K' \in \operatorname{Fun}(\operatorname{Mod}_A,\operatorname{Mod}_A)$.
\end{prop}

\begin{proof} The proof depends on the following result of Lurie (\cite[Thm.4.7.6.2]{HA}). 

Let $\mathcal{D}^{\bullet}$ be a cosimplicial $\infty$-category. If for every $[m] \to [n]$ in $\Delta$, the diagram 
$$\xymatrix{
\mathcal{D}^m \ar[d] \ar[r]^{d_0} &\mathcal{D}^{m+1} \ar[d]\\
\mathcal{D}^n \ar[r]^{d_0} &\mathcal{D}^{n+1} \\
}$$ is right adjointable, then 
\begin{enumerate}
\item The forgetful functor $U:\varprojlim \mathcal{D}^{\bullet} \to \mathcal{D}^0$ admits a right adjoint
\item The square $$\xymatrix{\varprojlim(\mathcal{D}^{\bullet}) \ar[d]_U \ar[r]^U &\mathcal{D}^0 \ar@/_1pc/[l]_H \ar[d]^{d^1}\\ \mathcal{D}^0 \ar[r]^{d^0} &\mathcal{D}^1 \ar@/_1pc/[l]_{H(0)}}$$ is right adjointable and there is an equivalence $U\circ H \simeq H(0) \circ d^1 \in \operatorname{Fun}(\mathcal{D}^0,\mathcal{D}^0)$
\item $U$ is comonadic
\end{enumerate}

Take $\mathcal{D}^{\bullet} = \operatorname{Cobar}(\omega)$. Assuming $\omega$ is compact, the cobar construction for $\omega$ is a cosimplicial symmetric monoidal stable $\infty$-category of the form

$$\xymatrix{
\operatorname{Mod}_A \ar[r] \ar@<1ex>[r] &\operatorname{Mod}_{\mathcal{H}} \ar[r] \ar@<1ex>[r] \ar@<-1ex>[r] &\operatorname{Mod}_{\mathcal{H}\otimes_A \mathcal{H}} \cdots\\
}$$
where $\mathcal{H}$ is the $A$-Hopf algebra representing $\operatorname{Aut}^{\otimes}_{\omega}$.

For every $[m] \to [n] \in \Delta$ the diagram in the theorem takes the form
$$\xymatrix{
\operatorname{Mod}_{\mathcal{H}^{\otimes_A^m}} \ar[r]^{d^0} \ar[d] &\operatorname{Mod}_{\mathcal{H}^{\otimes_A^{m+1}}} \ar[d] \\
\operatorname{Mod}_{\mathcal{H}^{\otimes_A^n}} \ar[r]^{d^0}  &\operatorname{Mod}_{\mathcal{H}^{\otimes_A^{n+1}}} \\
}$$ where every map is a base chage map. Therefore this is right ajointable. Therefore $U:\varprojlim \operatorname{Cobar}(\omega) \to \operatorname{Mod}_A$ admits a right adjoint $U^!$ and $U$ is comonadic with respect to the comonad $K' = U^!\circ U$. Furthermore, the diagram 
$$\xymatrix{
\varprojlim\operatorname{Cobar}(\omega) \ar[d]_U \ar[r]^U &\operatorname{Mod}_A \ar[d]^{G=-\otimes_A\mathcal{H}} \ar@/_1pc/[l]_{U^!}\\
\operatorname{Mod}_A \ar[r]_{G} &\operatorname{Mod}_{\mathcal{H}} \ar@/_1pc/[l]_{G'}\\
}$$ is right adjointable. 

Comparing with Prop.\ref{bc} there is an equivalence of $\infty$-comonads $$K \simeq K' \simeq G'\circ G \in \operatorname{Fun}(\operatorname{Mod}_A,\operatorname{Mod}_A).$$
\end{proof}

\begin{theorem}
Let $\omega$ be a compact fiber functor $\mathcal{C}^{\otimes} \to \operatorname{Mod}_A^{\otimes}$. Let $K$ be the $\infty$-comonad on $\operatorname{Mod}_A$ defined by $\omega$ and its right adjoint. Then there is an equivalence of $\infty$-categories $$\varprojlim \operatorname{Cobar}(\omega) \simeq \operatorname{LComod}_{K}(\operatorname{Mod}_A).$$
\end{theorem}

\begin{proof}
From Prop.\ref{equivcomonads} since $U$ is comonadic, the canonical map $\varprojlim\operatorname{Cobar}(\omega) \to \operatorname{LComod}_{K'}(\operatorname{Mod}_A)$ is an equivalence and since $K\simeq K'$ there is an equivalence $$\operatorname{LComod}_{K}(\operatorname{Mod}_A) \simeq \operatorname{LComod}_{K'}(\operatorname{Mod}_A).$$

$$\xymatrix{
\mathcal{C} \ar[r]_{\omega} \ar@{..>}[dr] &\operatorname{Mod}_A \ar@/_1pc/[l]_{\omega'} \ar@/_1pc/[d] \ar@(ur,ul)_{K} &\operatorname{Mod}_A  \ar@(ur,ul)_{K'} \ar@/_1pc/[d] \ar@/^1pc/[r]^{U^!} &\varprojlim(\operatorname{Cobar}(\omega)) \ar[l]^{U} \ar@{.>}[dl]^{\simeq}\\
&\operatorname{LComod}_K(\operatorname{Mod}_A)\ar[u] \ar[r]^{\simeq} &\operatorname{LComod}_{K'}(\operatorname{Mod}_A) \ar[u]\\
}$$

\end{proof}

\subsection{The comparison map}(Proof of main theorem Thm.\ref{theorem}) By \ref{repcobar}, Prop.\ref{equivcomonads} and Prop.\ref{comoduletodescentmap} there is an $\infty$-functor $\operatorname{Rep}_A(\operatorname{Aut}^{\otimes}_{\omega}) \to \operatorname{Desc}_T(\mathcal{C})$ by composition:

$$\xymatrix{\operatorname{Rep}_A(\operatorname{Aut}^{\otimes}_{\omega}) &\simeq &\varprojlim \operatorname{Cobar}(\omega) \ar[d]^{\simeq} \\
&&\operatorname{LComod}_K(\operatorname{Mod}_A) \ar[r]^{\Phi} &\operatorname{Desc}_{\mathcal{C}}(T)\\
}$$

\appendix

\section{Symmetric monoidal stable $\infty$-categories}

In this section we review sections of the theory of $\infty$-categories and algebras in the context of $\infty$-categories. We also define the notations used throughout the paper. We refer the reader to \cite{HTT} and \cite{HA} for reference.

\subsection{$\infty$-categories}

The notion of $\infty$-categories roughly captures the notion of topological categories where the composition and associativity properties are defined upto coherent homotopies. There are many models for this categorical structure. The notion of quasicategories developed by Joyal and Jardine and extensively used by Lurie (\cite{HTT},\cite{HA}) shall the basis for our work. 

\begin{definition}
A simplicial set $K$ is an {\em $\infty$-category} if it satisfies the following condition: for any $0<i<n$, any map $f_0:\Lambda^n_i \to K$ admits (possibly non-unique) extension $f:\Delta^n \to K$. Here $\Lambda_i^n \subseteq \Delta$ denotes the $i$-th horn, obtained from the simplex $\Delta^n$ by deleting the face opposite the $i$-th vertex.
\end{definition}

Let $K$ be a simplicial set underlying an $\infty$-category $\mathcal{C}$. The objects of $\mathcal{C}$ are the elements of $K_0$, the morphisms of $\mathcal{C}$ are the elements of $K_1$. The hom set $\operatorname{Maps}_{\mathcal{C}}(x,y)$ is a Kan complex. So every $\infty$-category has an underlying simplicial category.

A {\em functor} between $\infty$-categories is a map of simplicial sets. The functors betwen $\infty$-categories $\mathcal{C}$ and $\mathcal{D}$ assemble in an $\infty$-category $\operatorname{Fun}(\mathcal{C},\mathcal{D})$. We say a functor is an {\em equivalence of $\infty$-categories} when the map of the underlying simplicial categories is a Dwyer-Kan equivalence. The {\em homotopy category} of $\mathcal{C}$ is the homotopy category of the underlying simplicial category.

The notion of $\infty$-categories have the coherent homotopies built into the definition. Thus all functors are naturally derived and the notion of limits and colimits in the $\infty$-categorical context correspond to homotopy limits and colimits in older formulations.

\begin{definition}
Let $\operatorname{Cat}_{\infty}^{\Delta}$ be the simplicial category whose objects are small $\infty$-categories. Given two $\infty$-categories $\mathcal{C}$ and $\mathcal{D}$ define the mapping space $\operatorname{Maps}_{\operatorname{Cat}_{\infty}^{\Delta}}(\mathcal{C},\mathcal{D})$ to be the maximal Kan complex contained in the $\infty$-category of functors $\operatorname{Fun}(\mathcal{C},\mathcal{D})$. The $\infty$-category $\operatorname{Cat}_{\infty}$ is defined to be the simplicial nerve $N(\operatorname{Cat}_{\infty}^{\Delta})$.
\end{definition} 

The $\infty$-category $\operatorname{Cat}_{\infty}$ admits small limits. 

\subsection{Co-cartesian fibration of $\infty$-categories}
The notion of a co-cartesian fibration of $\infty$-categories $p:\mathcal{C} \to \mathcal{D}$ captures the idea of an $\infty$-category $\mathcal{C}$ fibered in $\infty$-categories over an $\infty$-category $\mathcal{D}$. Roughly this means that for every object $d\in \mathcal{D}$ the pre-image $p^{-1}(d) = \mathcal{C}_d \subseteq \mathcal{C}$ is an $\infty$-category, and for every edge $f:d_1 \to d_2$ in $\mathcal{D}$ there is a canonical functor $\mathcal{C}_{d_1} \to \mathcal{C}_{d_2}$ projecting to $f$ via $p$, upto higher homotopy coherences, resulting in a functor $F:\mathcal{D} \to \operatorname{Cat}_{\infty}$. The co-cartesian fibration $p:\mathcal{C} \to \mathcal{D}$ is said to be {\em classified} by the functor $F$. We give precise definitions here. 

\begin{definition}
Let $p:\mathcal{C} \to \mathcal{D}$ be a functor between $\infty$-categories. 
\begin{enumerate}
\item Given an edge $f:x\to y$ in $\mathcal{D}$, we say $f$ is {\em $p$-cartesian} if the canonical map $$\mathcal{C}_{/x} \to \mathcal{D}_{/p(x)} \times_{\mathcal{D}_{/p(y)}} \mathcal{C}_{/y}$$ is an equivalence of $\infty$-categories. The edge $f:x\to y$ is said to the {\em co-cartesian lift of $p(f)$ relative to $x$}.
\item The functor $p:\mathcal{C} \to \mathcal{D}$ is a {\em co-cartesian fibration} if every object in $$\operatorname{Fun}([1],\mathcal{D})\times_{s,\mathcal{D},p}\mathcal{C}$$ has a co-cartesian lift.
\end{enumerate}

There is a dual notion of cartesian fibrations. A functor $p:\mathcal{C} \to \mathcal{D}$ is {\em cartesian} if $p^{op}:\mathcal{C}^{op} \to \mathcal{D}^{op}$ is co-cartesian.
\end{definition}

The {\em $\infty$-category of co-cartesian fibrations over an $\infty$-category $\mathcal{D}$} is the subcategory of $(\operatorname{Cat}_{\infty})_{/\mathcal{D}}$ spanned by co-cartesian fibrations over $\mathcal{D}$. We denote this $\infty$-category by $\operatorname{co}\mathcal{CF}\operatorname{ib}(\mathcal{D})$.

There is a (contravariant) Grothendieck construction $$
\xymatrix{
\operatorname{Fun}(\mathcal{D},\operatorname{Cat}_{\infty}) \ar[r]_{\simeq}^{Gro} &\operatorname{co}\mathcal{CF}\operatorname{ib}(\mathcal{D})
}$$ which is an equivalence of $\infty$-categories. Given a functor $F:\mathcal{D} \to \operatorname{Cat}_{\infty}$ we list the notable features of the resulting co-cartesian fibration $p:\mathcal{C} \to \mathcal{D}$.

\begin{enumerate}
\item For $d\in \mathcal{D}$ the pre-image $p^{-1}(d) \subseteq \mathcal{C}$ is canonically equivalent to $F(b) \in \operatorname{Cat}_{\infty}$.
\item Given any edge $\phi:d_1\to d_2$ in $\mathcal{D}$ and an object $x \in p^{-1}(d_1)$ in the fiber over the source, there is a canonical morphism $x \to \phi_*x$ in $\mathcal{C}$ which projects to $\phi$ in $\mathcal{D}$, called the {\em $p$-cocartesian lift} of $\phi$ relative to $x$. Further, the identification $p^{-1}(d_2) \simeq F(d_2)$ can be used to identify the object $\phi_*(x) \in p^{-1}(d_2)$ with the object $(F\phi)(x) \in F(d_2)$.


\end{enumerate}

Dually there is an $\infty$-category $\mathcal{CF}\operatorname{ib}(\mathcal{D})$ of cartesian fibrations over $\mathcal{D}$, and a (covariant)  Grothendieck construction$$
\xymatrix{
\operatorname{Fun}(\mathcal{D}^{op},\operatorname{Cat}_{\infty}) \ar[r]_{\simeq}^{Gro} &\mathcal{CF}\operatorname{ib}(\mathcal{D})
}$$ which is an equivalence of $\infty$-categories.

\begin{remark}
A co-cartesian fibration $p:\mathcal{C} \to \mathcal{D}$ whose fibers $p^{-1}(d)$ are Kan complexes is called a {\em left fibration}. These form a full $\infty$-subcategory $\mathcal{LF}\operatorname{ib}(\mathcal{D}) \subseteq \operatorname{co}\mathcal{CF}\operatorname{ib}(\mathcal{D})$. The Grothendieck functor restricts to give an equivalence 
$$\xymatrix{
\operatorname{Fun}(\mathcal{D},\operatorname{Cat}_{\infty}) \ar[r]_{\simeq}^{Gro} &\operatorname{co}\mathcal{CF}\operatorname{ib}(\mathcal{D})\\
\operatorname{Fun}(\mathcal{D},\mathcal{S}) \ar[r]^{Gro}_{\simeq} \ar@{^{(}->}[u] &\mathcal{LF}\operatorname{ib}(\mathcal{D}) \ar@{^{(}->}[u]
}$$ from the $\infty$-category of functors between $\mathcal{D}$ and the $\infty$-category of spaces to the $\infty$-category of left fibrations over $\mathcal{D}$.

Dually, there is an the $\infty$-category of right fibrations $\mathcal{RF}\operatorname{ib}(\mathcal{D})$ which is equivalent via the  Grothendieck construction to the $\infty$-category $\operatorname{Fun}(\mathcal{D}^{op},\mathcal{S})$.
\end{remark}

\subsection{Monoids, groups and algebras}
A monoid object in an $\infty$-category corresponds to an $A_{\infty}$ object in the operadic context. The following defintion comes from the idea of modelling an $A_{\infty}$-monoid as a certain simplicial object. This goes back to Segal's notion of $\Gamma$-spaces.

\begin{definition}({\em Monoid objects}) Let $\mathcal{C}$ be an $\infty$-category which admits finite limits. A monoid object in $\mathcal{C}$ is a simplicial object $f: N(\Delta)^{op} \to \mathcal{C}$ having the property that $f([0])$ is a final object, and for each $n \in \mathbb{N}$, the inclusions $\{i-1,i\} \to [n]$ for $1\leq i \leq n$ induce the equivalence $$f([n]) \to f([1]) \times \ldots \times f([1])$$ where the right hand side in the $n$-fold product. Here $f([1])$ is thought of as underlying object of the monoid.

Denote by $\operatorname{Mon}(\mathcal{C})$ the full subcategory of $\operatorname{Fun}(N(\Delta)^{op},\mathcal{C})$ spanned by monoid objects. 
\end{definition}

\begin{definition}({\em Group objects}) Let $\mathcal{C}$ be an $\infty$-category which admits finite limits. A group object in $\mathcal{C}$ is a monoid object satisfying the following property: for every $n$ and every partition $[n] = A\cup B$ for which $A\cap B = \{s\}$, the square
$$\xymatrix{
f([n]) \ar[r] \ar[d] &f(A) \ar[d]\\
f(B) \ar[r] &f(\{s\})\\
}$$ is a pullback square in $\mathcal{C}$. Denote by $\operatorname{Grp}(\mathcal{C})$ the full subcategory of $\operatorname{Mon}(\mathcal{C})$ spanned by the group objects in $\mathcal{C}$.
\end{definition}

A commutative monoid in an $\infty$-category corresponds to a $E_{\infty}$-monoid in the traditional setting. The following definition is modelled on Segal's machine for infinite loop spaces. The idea is to replace $\Delta^{op}$ in the definition of a monoid with the category of pointed finite sets $\Gamma$.

\begin{definition}({\em Commutative monoids})\label{monoid}
For every $n\geq 0$ let $\langle n \rangle ^0 = \{1,\ldots,n\}$ and $\langle n \rangle = \langle n \rangle ^0 _* = \{*, 1,2, \ldots , n\}$ the pointed set obtained by adjoining the basepoint $*$ to $\langle n \rangle$. The category $\Gamma$ has as objects $\langle n \rangle$ and given  a morphism from $\langle m \rangle$ to $\langle n \rangle$ in $\Gamma$ is a map $\alpha: \langle m \rangle \to \langle n \rangle$  so that $\alpha(*) = *$.

Let $\mathcal{C}$ be an $\infty$-category closed under finite limits. A commutative monoid in $\mathcal{C}$ is a functor $f:N(\Gamma) \to \mathcal{C}$ so that $f\langle 0 \rangle$ is a final object and $f(\langle n \rangle) \simeq f(\langle 1 \rangle) \times \ldots \times f(\langle 1 \rangle)$.

\end{definition}

\begin{example} Let $f:X\to Y$ be a $1$-morphism in a $\infty$-category $\mathcal{C}$ which admits finite limits. Let $\Delta_{+}$ be the category of finite (including empty) linearly ordered sets. Denote by $[-1]$ the empty set. There is a fully faithful map $\Delta^1 \to \Delta_{+}$ taking $\{0\}$ to $[-1]$ and $\{1\}$ to $[0]$. Let $f:\Delta^1 \to \mathcal{C}$ be the map corresponding to $f:X \to Y$. Then the left Kan extension to $C(f):N(\Delta_{+}) \to \mathcal{C}$ exists and is called the \v{C}ech Nerve of $f$.

The {\em \v{C}ech Nerve} of $f$ is a group object in $\mathcal{C}_{/X}$.
\end{example}

\subsection{Monoidal $\infty$-categories and symmetric monoidal $\infty$-categories}

\begin{definition}
The $\infty$-category $\operatorname{Cat}_{\infty}$ is closed under finite limits. A {\em monoidal $\infty$-category} is a monoid object in $\operatorname{Cat}_{\infty}$. 
\end{definition}

\begin{remark}
Unwinding the defintion we see that a monoidal $\infty$-category is a simplicial $\infty$-category $$F:N(\Delta)^{op} \to \operatorname{Cat}_{\infty}$$  satisfying the conditions in Def.\ref{monoid}. This data is equivalent to a coCartesian fibration $p:\mathcal{C}^{\otimes} \to N(\Delta)^{op}$ of simplcial sets where $\mathcal{C}^{\otimes}_{[n]}$ is equivalent to $\mathcal{C}^{\otimes}_{[1]} \times \ldots \times \mathcal{C}^{\otimes}_{[1]}$ and $\mathcal{C}^{\otimes}_{[0]}$ is a final object. 

The $\infty$-category $\mathcal{C} = \mathcal{C}^{\otimes}_{[1]}$ is the underlying $\infty$-category of $\mathcal{C}^{\otimes}$. We say that $\mathcal{C}^{\otimes}$ is the {\em monoidal structure on $\mathcal{C}$}. Roughly a monoidal category $\mathcal{C}$ comes with a unit object $\Delta^0 \to \mathcal{C}$ and a product map $\otimes: \mathcal{C}\times \mathcal{C} \to \mathcal{C}$ which is associative upto coherent homotopies. 
\end{remark}
 
\begin{definition}({\em Algebras and Modules})

\begin{enumerate}
\item  An (associative) algebra object in a monoidal category $\mathcal{C}$ is a simplicial object $A: N(\Delta)^{op} \to \mathcal{C}$ so that $A([0])$ is a final object and $A([n]) \simeq A([1]) \otimes \ldots \otimes A([1])$. The $\infty$-category of {\em (associative) algebra objects in $\mathcal{C}$}, denoted by $\operatorname{Alg}(\mathcal{C})$ is the full subcategory of $\operatorname{Fun}(N(\Delta)^{op}, \mathcal{C})$ spanned by algebra objects in $\mathcal{C}$.

\item
An $\infty$-category $M$ is left tensored over a monoidal category $\mathcal{C}$ if there is an action $\mathcal{C} \times \mathcal{M} \to \mathcal{M}$ which is defined upto coherent homotopies. This is encoded as a simplicial object $f: N(\Delta)^{op} \to \operatorname{Cat}_{\infty}$ so that $f([n]) \simeq \mathcal{C}^{n} \times \mathcal{M}$ and $f([0]) \simeq \mathcal{M}$. 

\item
Given a $\infty$-category $\mathcal{M}$ which is left tensored over a monoidal $\infty$-category $\mathcal{C}$, there is an $\infty$-category of {\em left module objects in $\mathcal{M}$} denoted by $\operatorname{LMod}(\mathcal{M})$, It is the fullsubcategory of $\operatorname{Fun}(N(\Delta)^{op},\mathcal{M})$ spanned by simplicial objects $M:N(\Delta)^{op} \to \mathcal{M}$ where $M([n]) \simeq A^{\otimes_{n}}\otimes M$ and $M([0])\simeq M$. Here $A \in \operatorname{Alg}(\mathcal{C})$ and $M \in \mathcal{M}$. There is a map of $\infty$-categories $$\operatorname{LMod}(\mathcal{M}) \to \operatorname{Alg}(\mathcal{C}).$$ If $A \in \operatorname{Alg}(\mathcal{C})$ we let $\operatorname{LMod}_A(\mathcal{M})$ denote the the fiber $\operatorname{LMod}(\mathcal{M}) \times_{\operatorname{Alg}(\mathcal{C})}\{A\}$. We refer to $\operatorname{LMod}_A(\mathcal{M})$ as the $\infty$-category of {\em left $A$-modules in $\mathcal{M}$}.

\item Given an $\infty$-category $M$ which is right tensored over a monoidal category $\mathcal{C}$ and $A\in \operatorname{CAlg}(\mathcal{C})$, the $\infty$-categories $\operatorname{RMod}_{\mathcal{C}}(\mathcal{M})$ and $\operatorname{RMod}_{A}(\mathcal{M})$ are similarly defined. 

\end{enumerate}
\end{definition}

\begin{definition}({\em Coalgebras and comodules})
\begin{enumerate}
\item Let $\mathcal{C}$ be a monoidal $\infty$-category. Define $\operatorname{CoAlg}(\mathcal{C})$ to be $\operatorname{Alg}(\mathcal{C}^{op})^{op}$. We refer to this as the {\em $\infty$-category  of (coassociative) coalgebra objects in $\mathcal{C}$}.

\item Let $\mathcal{M}$ be an $\infty$-category left tensored over a monoidal $\infty$-category $\mathcal{C}$. Define $\operatorname{LComod}(\mathcal{M})$ to be $\operatorname{LMod}(\mathcal{M}^{op})^{op}$. We refer to this as the {\em $\infty$-category of (left) comodule objects of $\mathcal{M}$}. There is a map of $\infty$-categories $$\operatorname{LComod}(\mathcal{M}) \to \operatorname{CoAlg}(\mathcal{C}).$$ If $H \in \operatorname{CoAlg}(\mathcal{C})$, then we let $\operatorname{LComod}_H(\mathcal{M})$ denote the fiber $\operatorname{LComod}(\mathcal{M}) \times_{\operatorname{CoAlg}(\mathcal{C})}\{H\}$. We refer to $\operatorname{LComod}_H(\mathcal{M})$ as the {\em $\infty$-category of left $H$-comodules in $\mathcal{M}$}.
Alternately, $\operatorname{LComod}_H(\mathcal{M}) \simeq \operatorname{LMod}_H(\mathcal{M}^{op})^{op}$.

\end{enumerate}
\end{definition}

\begin{definition}
Let $\mathcal{C}$ be a monoidal $\infty$-category. Let $\mathcal{M}$ be right tensored over $\mathcal{C}$ and let $\mathcal{N}$ be left tensored over $\mathcal{C}$. Let $F: \mathcal{M} \times \mathcal{N} \to \mathcal{D}$ be a balanced pairing (\cite[]{DAG2}). Then there is a {\em two-sided bar construction} $$\operatorname{Bar}_{\bullet}:\operatorname{RMod}(\mathcal{M}) \times_{\operatorname{Alg}(\mathcal{C})}\operatorname{LMod}(\mathcal{N}) \subset \operatorname{Fun}(N(\Delta)^{op}, \mathcal{M} \times \mathcal{N}) \to \operatorname{Fun}(N(\Delta)^{op},\mathcal{D}).$$

If $\mathcal{D}$ admits geometric realizations of simplicial objets, the {\em relative tensor product} can be defined as the composition $$|-|\circ\operatorname{Bar}_{\bullet}: \operatorname{RMod}(\mathcal{M}) \times_{\operatorname{Alg}(\mathcal{C})}\operatorname{LMod}(\mathcal{N}) \to \mathcal{D}.$$
\end{definition}

\begin{remark}
Objects in $\operatorname{RMod}(\mathcal{M}) \times_{\operatorname{Alg}(\mathcal{C})}\operatorname{LMod}(\mathcal{N})$ can be identified with triples $(M,A,N)$ where $A$ is an algebra object in $\mathcal{C}$, $M$ is a right $A$-module and $N$ is a left $A$-module. Then the image of $(M,A,N)$ under the relative tensor product is denoted by $M\otimes_AN$.

Given a monoidal $\infty$-category $\mathcal{C}$ and $A \in \operatorname{Alg}(\mathcal{C})$, the relatve tensor product gives a pairing $$\operatorname{RMod}_A(\mathcal{C}) \times \operatorname{LMod}_A(\mathcal{C}) \to \mathcal{C}.$$
\end{remark}

\begin{definition}
A {\em symmetric monoidal $\infty$-category} is a commutative monoid object in $\operatorname{Cat}_{\infty}$. 
\end{definition}

\begin{remark} Unwinding the definition we can see that a symmetric monoidal structure on an $\infty$-category $\mathcal{C}$ is encoded by a functor $\mathcal{C}^{\otimes}:N(\Gamma) \to \operatorname{Cat}_{\infty}$, where $\mathcal{C}^{\otimes}_{\langle 0 \rangle}$ is a final object and $\mathcal{C}^{\otimes}_{\langle n \rangle} \simeq \mathcal{C}^{\otimes}_{\langle 1 \rangle} \times \ldots \times \mathcal{C}^{\otimes}_{\langle 1 \rangle}$ and $\mathcal{C}^{\otimes}_{\langle 1 \rangle} \simeq \mathcal{C}$.
\end{remark}

\begin{definition}
A commutative algebra in a symmetric monoidal category is a commutative monoid object with respect to the monoidal product in $\mathcal{C}$. This can be formulated as an $\infty$-functor $A:N(\Gamma) \to \mathcal{C}$ where $A(\langle n \rangle) \simeq A(\langle 1 \rangle)^{\otimes_n}$ and $A(\langle 0 \rangle)$ is a final object. The object $A(\langle 1 \rangle)$ can be thought of as the underlying algebra object of $A$.

The $\infty$-category of {\em commutative algebra objects in $\mathcal{C}$} denoted by $\operatorname{CAlg}(\mathcal{C})$ is the full subcategory of $\operatorname{Fun}(N(\Gamma), \mathcal{C})$ spanned by commuative algebra objects of $\mathcal{C}$.
\end{definition}

Given a symmetric monoidal category $\mathcal{C}$ and $A\in \operatorname{CAlg}(\mathcal{C})$, there is a category of commutative modules over $A$ denoted by $\operatorname{Mod}_A(\mathcal{C})$ (see \cite[]{HA}) which is equivalrnt to the $\infty$-category $\operatorname{LMod}_A(\mathcal{C})$ of left module objects over $A$ in $\mathcal{C}$. The $\infty$-category $\operatorname{Mod}_A(\mathcal{C})$ has the following notable features:

\begin{enumerate}

\item For every $A\in \operatorname{CAlg}(\mathcal{C})$ the $\infty$-category $\operatorname{LMod}_A(\mathcal{C})$ inherits a symmetric monoidal  structure given by the relative tensor product -$\otimes_A$-. This is encoded as a functor $\operatorname{LMod}_A(\mathcal{C})^{\otimes} \in \operatorname{Fun}(N(\Gamma), \operatorname{Cat}_{\infty})$. 
 
\item Let $f:A \to B \in \operatorname{CAlg}(\mathcal{C})$. Then the forgetful functor $\operatorname{Mod}_B(\mathcal{C}) \to \operatorname{Mod}_A(\mathcal{C})$ admits a left adjoint $M \mapsto M \otimes_AB$, which is a symmetric monoidal functor from $\operatorname{Mod}_A(\mathcal{C})$ to  $\operatorname{Mod}_B(\mathcal{C})$.

\end{enumerate}

\begin{definition} Given a commutative algebra $A$ in a symmetric monoidal category $\mathcal{C}$, the {\em $\infty$-category of $A$-algebra objects in $\mathcal{C}$} is defined to be the $\infty$-category $\operatorname{Alg}(\operatorname{LMod}_A(\mathcal{C}))$. Similarly, the $\infty$-category of {\em commutatitve $A$-algebras} in $\mathcal{C}$ defined to be $\operatorname{CAlg}(\operatorname{LMod}_A(\mathcal{C}))$.

We shall use the notations $\operatorname{Alg}_A(\mathcal{C})$ and $\operatorname{CAlg}_A(\mathcal{C})$ for these $\infty$-categories.
\end{definition}

\begin{definition}({\em Bi-algebras and Hopf algebras})
\begin{enumerate}
\item
A {\em commutative bi-algebra} in a symmetric monoidal category $\mathcal{C}$ is an object in $\operatorname{CoAlg}(\operatorname{CAlg}(\mathcal{C}))$. Given $A \in \operatorname{CAlg}(\mathcal{C})$, a {\em commutative bi-algebra over $A$} in $\mathcal{C}$ is an object in $\operatorname{CoAlg}(\operatorname{CAlg}_A)$.

We shall denote the $\infty$-categories by $\operatorname{BiAlg}(\mathcal{C})$ and $\operatorname{BiAlg}_A(\mathcal{C})$ respectively.

\item  Given a commutative bi-algebra $A$ in $\mathcal{C}$, this can be expressed as a functor $f:N(\Delta) \to \operatorname{CAlg}(\mathcal{C})$. We say that $A$ is a {\em commutative Hopf algebra object in $\mathcal{C}$} if $f^{op}:N(\Delta)^{op} \to \operatorname{Aff}(\mathcal{C})$ defines a group object in $\operatorname{Aff}(\mathcal{C})$. The $\infty$-category of {\em commutative Hopf algebra objects} in a symmetric monoidal category $\mathcal{C}$ is the full subcategory of $\operatorname{BiAlg}(\mathcal{C})$ spanned by Hopf algebra objects. Denote this $\infty$-category by $\operatorname{CHopf}(\mathcal{C})$. 

Given $A\in \operatorname{CAlg}(\mathcal{C})$, define the $\infty$-category of {\em commutative Hopf algebra objects over $A$} to be the full subcategory of $\operatorname{BiAlg}_A$ spanned by commutative Hopf algebra objects.

\item
Given a Bi-algebra $B \in \operatorname{CoAlg}(\operatorname{Alg}_A)$, there is an underlying commutative $A$-algebra $B$ with a comodule structure encoded as a cosimplicial object $N(\Delta) \to \operatorname{CAlg}_A$ of the form 
$$\xymatrix{
(B/A)^{\bullet} = (A \ar[r] \ar@<1ex>[r] &B \ar[r] \ar@<1ex>[r] \ar@<-1ex>[r] &B\otimes_AB \cdots )\\
}$$
Then, the $\infty$-category of comodules over the Bi-algebra $B$ over $A$ is 
$$\operatorname{Comod}_B(\operatorname{Mod}_A) \simeq \varprojlim \operatorname{Mod}_{(B/A)^{\bullet}}.$$
\end{enumerate}
\end{definition}

\subsection{Presentable $\infty$-categories}

\begin{definition}(\cite[5.5]{HTT})
An $\infty$-category is {\em presentable} if it is closed under all small colimits (also limits by Prop.5.5.2.4 \cite{HTT} and more over are generated in a weak sense by a small category (accessible). Presentable $\infty$-categories form an $\infty$-category $Pr^L$ whose morphism are continous funtors, i.e. functors that preserve all small colimits, $\operatorname{Map}_{Pr^L}(\mathcal{C},\mathcal{D})= \operatorname{Fun}^L(\mathcal{C},\mathcal{D})$. 
\end{definition}

The $\infty$-category $\widehat{\mathcal{C}}_{\infty}$ is monoidal via the cartesian product. The subcategory $Pr^L \subset \widehat{\mathcal{C}}_{\infty}$ obtains a symmetric monoidal structure $$\otimes:Pr^L \times Pr^L \to Pr^l.$$ The tensor product $\mathcal{C}\otimes \mathcal{D}$ of $\mathcal{C}$,$\mathcal{D}$ presentable $\infty$-categories is the universal recipient of a functor from the cartesian product $\mathcal{C} \times \mathcal{D}$ which preserves colimits in each variable separately. $\mathcal{C}\otimes\mathcal{D}$ is defined to be $\operatorname{Fun}^R(\mathcal{C}^{op},\mathcal{D})$. The unit object of the monoidal structure in $Pr^L$ is $\mathcal{S}$, the $\infty$-category of spaces. Every presentable $\infty$-category is tensored over $\mathcal{S}$.

\begin{definition}
An object in $\operatorname{Alg}(Pr^{L})$ corresponds to a monoidal $\infty$-category $\mathcal{C}^{\otimes}$ whose underlying category $\mathcal{C}^{\otimes}_{[1]} \simeq \mathcal{C}$ is presentable and the product $\mathcal{C}\times\mathcal{C} \to \mathcal{C}$ preserves colimits separately in each variable. We shall call $\operatorname{Alg}(Pr^{L})$ the $\infty$-category of {\em presentable monoidal categories}.
\end{definition}

\begin{definition}
The $\infty$-category $Pr^L$ is both left and right tensored over $Pr^L$ by the tensor product of presentable $\infty$-categories. Denote by $\operatorname{LMod}(Pr^L)$ and $\operatorname{RMod}(Pr^L)$ the $\infty$-categories of left and right module objects respectively.

Given a presentable monoidal $\infty$-category $\mathcal{C}$, denote by $$\operatorname{LMod}_{\mathcal{C}}(Pr^L) = \operatorname{LMod}(Pr^L)\times_{\operatorname{Alg}(\mathcal{C})}\{\mathcal{C}\}$$ and, $$\operatorname{RMod}_{\mathcal{C}}(Pr^L) = \operatorname{RMod}(Pr^L)\times_{\operatorname{Alg}(\mathcal{C})}\{\mathcal{C}\}$$ the $\infty$-categories of {\em left  and right modules over $\mathcal{C}$} respectively.
\end{definition}

\begin{remark}

Let $\mathcal{C} \in \operatorname{Alg}(Pr^L)$. There is relative tensor product $$\operatorname{RMod}_{\mathcal{C}}(Pr^L) \times \operatorname{LMod}_{\mathcal{C}}(Pr^L) \to Pr^L$$ defined using the two-sided bar construction.

\begin{enumerate}

\item If $\mathcal{C} \in \operatorname{CAlg}(Pr^L)$ is a symmetric monoidal presentable $\infty$-category, then $\operatorname{LMod}_{\mathcal{C}}(Pr^L)$ gets a symmetric monoidal structure via the relative tensor product,$$ \operatorname{LMod}_{\mathcal{C}}(Pr^L) \times \operatorname{LMod}_{\mathcal{C}}(Pr^L) \to \operatorname{LMod}_{\mathcal{C}}(Pr^L).$$ Given $\mathcal{A},\mathcal{B}$ two presentable $\infty$-categories left tensored over a symmetric monoidal $\infty$-category $\mathcal{C}$, we denote the image under the realtive tensor product by $\mathcal{A}\otimes_{\mathcal{C}}\mathcal{B}$.
\item Given a map $\mathcal{C} \to \mathcal{D}$ of symmetric monoidal presentable $\infty$-categories, there is a symmetric monoidal functor $\operatorname{LMod}_{\mathcal{C}}(Pr^L) \to \operatorname{LMod}_{\mathcal{D}}(Pr^L)$ which on objects is $\mathcal{A} \mapsto \mathcal{A}\otimes_{\mathcal{C}}\mathcal{D}$, and is right adjoint to the forgetful functor.
\end{enumerate}
\end{remark}

\begin{definition}
Given a symmetric monoidal presentable $\infty$-category $\mathcal{C}$, the $\infty$-category of {\em algebra objects over $\mathcal{C}$} is the $\infty$-category $\operatorname{Alg}(\operatorname{LMod}_{\mathcal{C}}(Pr^L))$. Similarly, the $\infty$-category of {\em commutative algebra objects over $\mathcal{C}$} is given by $\operatorname{CAlg}(\operatorname{LMod}_{\mathcal{C}}(Pr^L))$. We shall denote these $\infty$-categories by $\operatorname{Alg}_{\mathcal{C}}$ and $\operatorname{CAlg}_{\mathcal{C}}$ respectively.
\end{definition}

Denote by $Pr^{L,\sigma}$ the full subcategory of $Pr^L$ spanned by stable $\infty$-categories. The symmetric monoidal structure on $Pr^L$ restricts to one on the full subcategory $Pr^{L,\sigma}$. The unit for the monoidal structure is $\operatorname{Sp}$ the stable $\infty$-category of spectra. 

\begin{definition}
An object in $\operatorname{CAlg}(Pr^{L,\sigma})$ corresponds to a symmetric monoidal $\infty$-category $\mathcal{C}^{\otimes}$ whose underlying category $\mathcal{C}^{\otimes}_{[1]} \simeq \mathcal{C}$ is presentable and stable, and the product $\mathcal{C}\times\mathcal{C} \to \mathcal{C}$ preserves colimits separately in each variable. We shall call $\operatorname{CAlg}(Pr^{L,\sigma})$ the $\infty$-category of {\em symmetric monoidal presentable stable $\infty$-categories}.
\end{definition}
\begin{remark}
\begin{enumerate}
\item[]
\item
The $\infty$-category $\operatorname{Alg}(Pr^{L,\sigma})$ has a unit object which is equivalent to the unit object of $Pr^L$ under the tensor monoidal structure. This produces a monoidal structure on spectra:
$$\wedge:\operatorname{Sp} \times \operatorname{Sp} \to \operatorname{Sp}$$ which is called the smash product monoidal structure. The algebra and commutative algebra objects of $\operatorname{Sp}$ with respect to $\wedge$ are exactly the classical $A_{\infty}$-ring spectra and $E_{\infty}$-ring spectra.
\item
Every presentable stable $\infty$-category is canonically tensored over $\operatorname{Sp}$.
\end{enumerate}
\end{remark}

\subsection{(Co)monads in $\infty$-categories}
\begin{definition}((Co)monads and (co)modules)
Given an $\infty$-category $\mathcal{D}$, the $\infty$-category of functors $\operatorname{Fun}(\mathcal{D},\mathcal{D})$ is monoidal and $\mathcal{D}$ is left tensored over $\operatorname{Fun}(\mathcal{D},\mathcal{D})$. 

\begin{enumerate}
\item A functor $K \in \operatorname{Fun}(\mathcal{D},\mathcal{D})$ is a {\em (co)monad} if $K \in \operatorname{(Co)Alg}(\operatorname{Fun}(\mathcal{D},\mathcal{D}))$.
\item There is an $\infty$-category $\operatorname{L(Co)mod}_K(\mathcal{D})$ of (co)modules over a (co)monad $K$ in $\mathcal{D}$.
\end{enumerate}
\end{definition}

There is a natural forgetful map $U_K:\operatorname{L(Co)mod}_K(\mathcal{D}) \to \mathcal{D}$. 

\begin{remark} Informally, a monad $T$ on an $\infty$-category $\mathcal{C}$ is an endofunctor $T:\mathcal{C}\to \mathcal{C}$ equipped with maps $1 \to T$ and $T\circ T \to T$ which satisfies the usual unit and associativity conditions up to coherent homotopy. A module over the monad $T$ is an object $x\in \mathcal{C}$ equipped with a structure map $\eta: T(x) \to x $ which is compatible with the algebra structure on $T$, again up to coherent homotopy. The forgetful map takes a module to the underlying object in $\mathcal{C}$.\end{remark}

\begin{prop}(see \cite[Prop. 4.7.4.3]{HA})\label{composition}
Given a functor $F:\mathcal{C} \to \mathcal{D}$ of $\infty$-categories which admits a right adjoint $G$. Then the composition $K=F\circ G \in \operatorname{Fun}(\mathcal{D},\mathcal{D})$ is a comonad on $\mathcal{D}$ and $T=G\circ F$ is a monad on $\mathcal{C}$.

There are canonial maps $F':\mathcal{C} \to \operatorname{LComod}_K(\mathcal{D})$ and  $G':\mathcal{D} \to \operatorname{LMod}_T(\mathcal{C})$ such that $U_K\circ F' \simeq F \in \operatorname{Fun}(\mathcal{D},\mathcal{D})$ and $U_T\circ G' \simeq G \in \operatorname{Fun}(\mathcal{D})$,
\end{prop}

\begin{remark} In ordinary categorical setting it is easy to check that the composition $T$ is a monad on $\mathcal{C}$. However, as Lurie notes in \cite[Remark 4.7.0.4]{HA}, this a not so straightforward in the $\infty$-categorical setting. In order to give a algebra structure on the composition $T = G\circ F \in \operatorname{Fun}(\mathcal{C}, \mathcal{C})$ it is not enough to give a produce a single natural transformation $T\circ T \to T$ but an infinite system of coherence data, which is not easy to describe explicitly.
\end{remark}

\begin{definition} Let $F:\mathcal{C} \to \mathcal{D}$ be a map of $\infty$-categories that admits a right adjoint $G$ and let $K=F\circ G$ be the composition comonad on $\mathcal{D}$ and $T = G\circ F$ be the composition monad on $\mathcal{C}$. Then $F$ is said to be {\em comonadic} if the comparison map $F':\mathcal{C} \to \operatorname{LComod}_K(\mathcal{D})$ is an equivalence of $\infty$-categories, and $G$ is said to be {\em  monadic} if the comparison map $G': \mathcal{D} \to \operatorname{LMod}_T(\mathcal{C})$ is an equivalence of $\infty$-categories.
\end{definition}

\end{document}